\documentclass[11pt]{amsart}

\usepackage[utf8]{inputenc}
\usepackage{amsfonts}
\usepackage{amsmath}
\usepackage{amssymb}
\usepackage{amsthm}
\usepackage{mathrsfs}
\usepackage{stmaryrd}
\usepackage{color}
\usepackage[english]{babel}
\usepackage{fontenc}
\usepackage{url}
\usepackage{graphicx}
\usepackage{hyperref}
\usepackage{caption}
\usepackage{epstopdf}
\usepackage{float}
\usepackage{indentfirst}
\usepackage[export]{adjustbox}
\usepackage{tikz}
\usetikzlibrary{matrix,arrows,decorations.pathmorphing}

\allowdisplaybreaks

\newtheorem{theorem}{Theorem}[section]
\newtheorem{lemma}[theorem]{Lemma}
\newtheorem{proposition}[theorem]{Proposition}
\newtheorem{corollary}[theorem]{Corollary}
\newtheorem*{theorem*}{Theorem}
\newtheorem*{liftinglemma}{Lifting Lemma}

\theoremstyle{definition}

\newtheorem{remark}[theorem]{Remark}
\newtheorem*{notation}{Notation}
\newtheorem*{combinatorial description of boundary divisors}{Combinatorial description of boundary divisors}
\newtheorem*{combinatorial description of equivalence classes of F-curves}{Combinatorial description of equivalence classes of F-curves}
\newtheorem*{classification in characteristic 2}{Classification in characteristic $2$}
\newtheorem*{properties of the keel-vermeire divisors on M06bar}{Properties of the Keel-Vermeire divisors on $\overline{M}_{0,6}$}
\newtheorem*{fulton'squestion}{Question}
\newtheorem*{question1}{Question 1}
\newtheorem*{question2}{Question 2}
\newtheorem*{question3}{Question 3}
\newtheorem*{acknowledgements}{Acknowledgements}

\newtheorem{definition}[theorem]{Definition}

\newcommand{\peff}{\overline{\textrm{Eff}}}
\newcommand{\eff}{\textrm{Eff}}
\author{Luca Schaffler}
\thanks{Partially supported by the Office of the Graduate School of the University of Georgia.}
\title{On the cone of effective 2-cycles on $\overline{M}_{0,7}$}

\begin{document}
\maketitle
\begin{abstract}
Fulton's question about effective $k$-cycles on $\overline{M}_{0,n}$ for $1<k<n-4$ can be answered negatively by appropriately lifting to $\overline{M}_{0,n}$ the Keel-Vermeire divisors on $\overline{M}_{0,k+1}$. In this paper we focus on the case of $2$-cycles on $\overline{M}_{0,7}$, and we prove that the $2$-dimensional boundary strata together with the lifts of the Keel-Vermeire divisors are not enough to generate the cone of effective $2$-cycles. We do this by providing examples of effective $2$-cycles on $\overline{M}_{0,7}$ that cannot be written as an effective combination of the aforementioned $2$-cycles. These examples are inspired by a blow up construction of Castravet and Tevelev.
\end{abstract}


\section*{Introduction}
An open problem in the birational geometry of $\overline{M}_{0,n}$, the moduli space of stable $n$-pointed rational curves, is the F-conjecture. This conjecture claims that the cone $\eff_1(\overline{M}_{0,n})$ of effective curves, is generated by the numerical equivalence classes of $1$-dimensional boundary strata, which are obtained by intersecting boundary divisors. This is known to be true if $n\leq7$ (see \cite{keelmckernan}).

A similar question (which is known as Fulton's question) was stated in \cite{keelmckernan} also for the cone $\eff_k(\overline{M}_{0,n})$ of effective $k$-cycles with $1<k<n-3$:
\begin{center}
\noindent\emph{Is the cone $\emph{Eff}_k(\overline{M}_{0,n})$ generated by the $k$-dimensional boundary strata?}
\end{center}
\noindent Denote by $V_k(\overline{M}_{0,n})$ the cone generated by the numerical equivalence classes of the $k$-dimensional boundary strata. Then the question is whether or not $\eff_k(\overline{M}_{0,n})$ is equal to $V_k(\overline{M}_{0,n})$. As Keel and Vermeire pointed out in the case of divisors (see \cite{gibneykeelmorrison}, \cite{vermeire}), the cone $V_{n-4}(\overline{M}_{0,n})$ is strictly contained in $\eff_{n-4}(\overline{M}_{0,n})$, and one can see that $V_k(\overline{M}_{0,n})\subsetneq\eff_k(\overline{M}_{0,n})$ for all $1<k<n-4$ by appropriately lifting to $\overline{M}_{0,n}$ the Keel-Vermeire divisors on $\overline{M}_{0,k+1}$ (see Section~\ref{lift!}, in particular Corollary~\ref{intermediatecasesout}). So the problem is to understand what lies in $\eff_k(\overline{M}_{0,n})\setminus V_k(\overline{M}_{0,n})$ (see \cite{hassetttschinkel}, \cite{castravet}, \cite{castravettevelev13}, \cite{dorangiansiracusajensen}, \cite{opie} for the codimension $1$ case). Recently, a lot of work has been done in order to understand the cones of effective and pseudoeffective cycles of higher codimension on projective varieties (see \cite{debarreeinlazarsfeldvoisin}, \cite{fulger}, \cite{tarasca}, \cite{lehmann}, \cite{fulgerlehmann} and \cite{chencoskun}).

We work over an algebraically closed field $\mathbb{K}$ of any characteristic. The main result of this paper (Theorem~\ref{firstmaintheorem}) can be synthesized in the following statement
\begin{theorem*}
The $2$-dimensional boundary strata on $\overline{M}_{0,7}$ together with the lifts of the Keel-Vermeire divisors on $\overline{M}_{0,6}$ are not enough to generate the cone $\emph{Eff}_2(\overline{M}_{0,7})$.
\end{theorem*}
The lifts of the Keel-Vermeire divisors are defined as the pushforwards with respect to the natural inclusion $D_{ab}\hookrightarrow\overline{M}_{0,7}$ of the Keel-Vermeire divisors on the boundary divisor $D_{ab}$ (which is isomorphic to $\overline{M}_{0,6}$) for any $\{a,b\}\subset\{1,\ldots,7\}$. In this way we produce $315$ extremal rays of $\eff_2(\overline{M}_{0,7})$ which lie outside of $V_2(\overline{M}_{0,7})$ (see Proposition~\ref{315extremalrays}). Denote with $V_2^{KV}(\overline{M}_{0,7})$ the cone generated by $V_2(\overline{M}_{0,7})$ and by these lifts.

Examples of effective $2$-cycles on $\overline{M}_{0,7}$ whose numerical equivalence classes do not lie in the cone $V_2^{KV}(\overline{M}_{0,7})$ are produced using the following blow up construction of Castravet and Tevelev (see \cite[Theorem 3.1]{castravettevelev12}): take seven labeled points in $\mathbb{P}^2$ which do not lie on a (possibly reducible) conic. Then the blow up of $\mathbb{P}^2$ at these points can be embedded in $\overline{M}_{0,7}$ as an effective $2$-cycle. Using this construction and considering particular arrangements of seven labeled points in $\mathbb{P}^2$, we define what we call \emph{special hypertree surfaces} on $\overline{M}_{0,7}$ (see Definition~\ref{definitionofspecialhypertreesurface}), which are related to Castravet and Tevelev hypertrees (see \cite{castravettevelev13}). In Theorem~\ref{firstmaintheorem} we prove that the numerical equivalence class of a special hypertree surface does not lie in the cone $V_2^{KV}(\overline{M}_{0,7})$. This implies that $V_2^{KV}(\overline{M}_{0,7})\subsetneq\eff_2(\overline{M}_{0,7})$, which is our main result. An example of $7$-points arrangement in $\mathbb{P}^2$ which gives rise to a special hypertree surface on $\overline{M}_{0,7}$ is the one shown in Figure~\ref{figure1}.

All the other special hypertree surfaces are obtained by permuting the labels of the points arrangement in Figure~\ref{figure1}. In Section~\ref{classificationspecialhypertreesurfaces} we show that there are $210$ (resp. $30$) distinct numerical equivalence classes of special hypertree surfaces on $\overline{M}_{0,7}$ if the characteristic of the base field is different from $2$ (resp. equal to $2$).

Summing up, if we denote with $V_2^{KV+CT}(\overline{M}_{0,7})$ the cone generated by $V_2^{KV}(\overline{M}_{0,7})$ and by the numerical equivalence classes of the embedded blow ups of $\mathbb{P}^2$ at seven points, we have the following chain of containments
\begin{equation*}
V_2(\overline{M}_{0,7})\subsetneq V_2^{KV}(\overline{M}_{0,7})\subsetneq V_2^{KV+CT}(\overline{M}_{0,7})\subseteq\eff_2(\overline{M}_{0,7}).
\end{equation*}
\begin{center}
\begin{figure}
\centering
\includegraphics[scale=.85,valign=t]{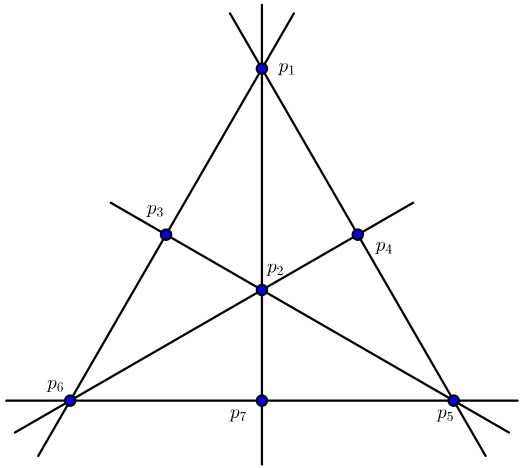}
\caption{7-points arrangement in $\mathbb{P}^2$ which gives a special hypertree surface on $\overline{M}_{0,7}$ $^{\ref{FootNoteForFigureCaption}}$.}
\label{figure1}
\end{figure}
\end{center}
\addtocounter{footnote}{1}\footnotetext{All the figures in this paper were realized using the software GeoGebra, Copyright \copyright International GeoGebra Institute, 2013.\label{FootNoteForFigureCaption}}
The second main result of this paper is an explicit description of the intersection theory of the $2$-dimensional boundary strata on $\overline{M}_{0,7}$. In Proposition~\ref{intersectiondistinct} and Proposition~\ref{intersectionsame} we give formulas that compute the intersection number of two $2$-dimensional boundary strata on $\overline{M}_{0,7}$. Then we study the numerical equivalence classes of these $2$-cycles (see Propositions~\ref{equivalenceclassesmadeclear} and \ref{distinctrays}), and this, together with some recent results of Chen and Coskun in \cite{chencoskun}, allows us to give a complete description of the cone $V_2(\overline{M}_{0,7})$ (see Corollary \ref{IseeV2}). We also fully describe the bilinear form $N_2(\overline{M}_{0,7})\times N_2(\overline{M}_{0,7})\rightarrow\mathbb{R}$ given by the intersection product (see Propositions~\ref{dimension127}~and~\ref{signature8641}).

In Section~\ref{preliminaries} we recall some basic facts and notations about $\overline{M}_{0,n}$ that are used in this paper. Section~\ref{vital2cyclesandintersections} contains the formulas for the intersection of two $2$-dimensional boundary strata on $\overline{M}_{0,7}$, and the complete study of the cone $V_2(\overline{M}_{0,7})$. In Section~\ref{lift!} there is a detailed description of the lifting technique, which is immediately applied in Section~\ref{liftsofthekeelvermeiredivisors} to describe the lifts to $\overline{M}_{0,7}$ of the Keel-Vermeire divisors on $\overline{M}_{0,6}$. Section~\ref{embeddedblowupsofp2inm0nbar} is where we discuss the embedded blow ups of $\mathbb{P}^2$ in $\overline{M}_{0,7}$ and where we prove our main theorem. In Section~\ref{generalization} we generalize the construction of the two cones $V_2^{KV}(\overline{M}_{0,7})$ and $V_2^{KV+CT}(\overline{M}_{0,7})$ to any $\overline{M}_{0,n}$ for $n>7$. We also state some questions that will be the object of further investigation.

\

\begin{acknowledgements}
I would like to express my gratitude to my advisor, Valery Alexeev, for his insightful comments and helpful discussions. I am also grateful to Angela Gibney for her great suggestions and support. Many thanks to Noah Giansiracusa, Daniel Krashen, Dino Lorenzini and Robert Varley for interesting discussions and for their helpful feedback. My gratitude also goes to Ana-Maria Castravet, Dawei Chen, Izzet Coskun and Jenia Tevelev for great discussions related to this paper. A special thanks to Ana-Maria Castravet for pointing out a mistake in the first version of this paper, and to Jenia Tevelev for a helpful discussion on Lemma~\ref{tevelevmihadetto}. I am also grateful to the referees for their careful reading of the paper and their comments.
\end{acknowledgements}


\tableofcontents


\section{Preliminaries: boundary strata on $\overline{M}_{0,n}$}
\label{preliminaries}
In this section we review some of the main definitions and facts about the boundary strata on $\overline{M}_{0,n}$. For a more detailed discussion, see for example \cite{keelmckernan}. Equivalence between $k$-cycles on $\overline{M}_{0,n}$ refers to numerical equivalence, which is the same as rational equivalence and algebraic equivalence by \cite{keel}.
\begin{definition}
\label{definitionvitalcycle}
The irreducible components of the locus of points on $\overline{M}_{0,n}$ parametrizing stable $n$-pointed rational curves with at least $n-3-k$ nodes, have dimension $k$ and are called \emph{boundary $k$-strata}. Codimension $1$ (resp. $1$-dimensional) boundary strata are also called \emph{boundary divisors} (resp. \emph{F-curves}).
\end{definition}
\begin{definition}
Given $n\geq3$ and $0\leq k\leq n-3$, define $V_k(\overline{M}_{0,n})$ to be the cone generated by the equivalence classes of the boundary $k$-strata on $\overline{M}_{0,n}$ ($V$ stands for ``vital cycles", as they were called in \cite{keelmckernan}).
\end{definition}
\begin{notation}
If $n$ is a positive integer, then $[n]$ denotes the set $\{1,\ldots,n\}$.
\end{notation}
\begin{combinatorial description of boundary divisors}
There is a bijection between boundary divisors and partitions $I\amalg I^c=[n]$, with $2\leq|I|\leq n-2$. $D_{I}=D_{I^c}$ denotes the boundary divisor corresponding to the partition $I\amalg I^c=[n]$. $\delta_{I}=\delta_{I^c}$ denotes the equivalence class of $D_{I}$. For simplicity, the equivalence class of a boundary divisor will be called just boundary divisor.
\end{combinatorial description of boundary divisors}
\begin{combinatorial description of equivalence classes of F-curves}
There is a bijection between equivalence classes of F-curves and partitions of $[n]=\{1,\ldots,n\}$ into four nonempty subsets (see \cite[Lemma 4.3]{keelmckernan}). Given a partition $I_1\amalg I_2\amalg I_3\amalg I_4=[n]$, we denote by $F_{I_1,I_2,I_3,I_4}$ the equivalence class of the F-curves corresponding to that partition.
\end{combinatorial description of equivalence classes of F-curves}
Every boundary stratum on $\overline{M}_{0,n}$ can be realized as the complete intersection of all the boundary divisors containing it as follows. Let $B$ be a boundary stratum and let $C(B)$ be the stable $n$-pointed rational curve corresponding to the generic point of $B$ ($C(B)$ has as many nodes as the codimension of $B$). If $\textrm{Sing}(C(B))$ denotes the set of singular points of $C(B)$, given $p\in\textrm{Sing}(C(B))$ let $T_p$ be the set of markings that are over one of the two connected components of the normalization of $C(B)$ at $p$. Then we have that
\begin{equation*}
B=\bigcap_{p\in\textrm{Sing}(C(B))}D_{T_p}.
\end{equation*}
Moreover, since the boundary of $\overline{M}_{0,n}$ has normal crossings, we have that the equivalence class of $B$ is the product of all the $\delta_{T_p}$ as $p$ varies among the nodes of $C(B)$.

\

The last thing we want to recall is \cite[Fact 4]{keel}: given two boundary divisors $D_I$, $D_J$ on $\overline{M}_{0,n}$, then $D_I\cap D_J\neq\emptyset\Leftrightarrow I**J$, which by definition means
\begin{equation*}
I\subseteq J~\textrm{or}~I\subseteq J^c~\textrm{or}~I\supseteq J~\textrm{or}~I\supseteq J^c.
\end{equation*}


\section{The cone of boundary $2$-strata on $\overline{M}_{0,7}$}
\label{vital2cyclesandintersections}
The main object of our study is $\eff_2(\overline{M}_{0,7})$, which is a subcone of the real vector space $N_2(\overline{M}_{0,7})$ (in Section~\ref{studyofthebilinearform} we show that $\dim_{\mathbb{R}}N_2(\overline{M}_{0,7})=127$). We start by analyzing the subcone $V_2(\overline{M}_{0,7})\subseteq\eff_2(\overline{M}_{0,7})$. The first thing we want to do is to give a combinatorial description of the boundary $2$-strata on $\overline{M}_{0,7}$. After this, we study their intersections and their equivalence classes.
\subsection{Combinatorial description of the boundary $2$-strata on $\overline{M}_{0,7}$}
According to Definition~\ref{definitionvitalcycle}, a boundary $2$-stratum on $\overline{M}_{0,7}$ is the closure of the locus of points parametrizing stable $7$-pointed rational curves of the shape shown in Figure~\ref{figure2}, where $I\amalg J\amalg K$ is a given partition of $[7]$.
\begin{center}
\begin{figure}[H]
\centering
\includegraphics[scale=.58]{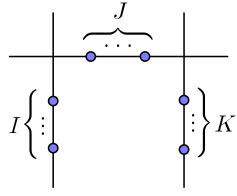}
\caption{Stable $7$-pointed rational curve parametrized by the generic point of a boundary $2$-stratum.}
\label{figure2}
\end{figure}
\end{center}
Stability imposes that $2\leq|I|\leq4$, $1\leq|J|\leq3$ and $2\leq|K|\leq4$. Therefore

\

\noindent\emph{there is a bijection between set-theoretically distinct boundary $2$-strata, and partitions $I\amalg J\amalg K$ of $[7]$, with $2\leq|I|\leq4$, $1\leq|J|\leq3$ and $2\leq|K|\leq4$, modulo the equivalence relation $I\amalg J\amalg K\sim K\amalg J\amalg I$.}

\

\noindent With $s_{I,J,K}\subset\overline{M}_{0,7}$ we denote the boundary $2$-stratum corresponding to the partition $I\amalg J\amalg K$ of $[7]$. The equivalence class of $s_{I,J,K}$ is denoted by $\sigma_{I,J,K}$. Obviously, we have that $\sigma_{I,J,K}=\delta_I\cdot\delta_K$. An easy combinatorial count tells us that there are $490$ set-theoretically distinct boundary $2$-strata $s_{I,J,K}$. A similar description applies for codimension $2$ boundary strata on $\overline{M}_{0,n}$ for $n\geq8$. For general results about boundary strata of codimension $2$ on $\overline{M}_{0,n}$, see \cite[Section 6]{chencoskun}.
\subsection{Intersection of two distinct boundary $2$-strata}
Given $\sigma_{I,J,K}$ and $\sigma_{L,M,N}$, our goal is to compute the intersection $\sigma_{I,J,K}\cdot\sigma_{L,M,N}=\delta_I\cdot\delta_K\cdot\delta_L\cdot\delta_N$. This intersection is clearly zero, unless we require that the condition defined here below is satisfied.
\begin{definition}
Consider two boundary $2$-strata $s_{I,J,K}$ and $s_{L,M,N}$. Assume that
\begin{equation*}
I**L~\textrm{and}~I**N~\textrm{and}~K**L~\textrm{and}~K**N.
\end{equation*}
If this condition is satisfied, we write $s_{I,J,K}**s_{L,M,N}$.
\end{definition}
\begin{lemma}
\label{avolteritornano}
Let $D_{I_1},D_{I_2}$ and $D_{I_3}$ be three distinct boundary divisors on $\overline{M}_{0,7}$ such that $I_a**I_b$ for all $\{a,b\}\subset\{1,2,3\}$. Then $D_{I_1}\cap D_{I_2}\cap D_{I_3}$ is an F-curve.
\end{lemma}
\begin{proof}
Assume without loss of generality that $I_1\cap I_2=\emptyset$. We know that $I_3**I_1$ and $I_3**I_2$, therefore
\begin{align*}
&(I_3\subset I_1~\textrm{or}~I_3\subset I_1^c~\textrm{or}~I_3\supset I_1~\textrm{or}~I_3\supset I_1^c)~\textrm{and}\\
&(I_3\subset I_2~\textrm{or}~I_3\subset I_2^c~\textrm{or}~I_3\supset I_2~\textrm{or}~I_3\supset I_2^c).
\end{align*}
Among these $16$ cases, the only possible are
\begin{align*}
(I_3\subset I_1~\textrm{and}~I_3\subset I_2^c)~&\textrm{or}~(I_3\subset I_1^c~\textrm{and}~I_3\subset I_2)~\textrm{or}\\
(I_3\subset I_1^c~\textrm{and}~I_3\subset I_2^c)~&\textrm{or}~(I_3\subset I_1^c~\textrm{and}~I_3\supset I_2)~\textrm{or}\\
(I_3\supset I_1~\textrm{and}~I_3\subset I_2^c)~&\textrm{or}~(I_3\supset I_1~\textrm{and}~I_3\supset I_2)~\textrm{or}\\
(I_3\supset I_1~\textrm{and}~I_3\supset I_2^c)~&\textrm{or}~(I_3\supset I_1^c~\textrm{and}~I_3\supset I_2).
\end{align*}
Up to changing $I_3$ with $I_3^c$, we just need to consider
\begin{align*}
(I_3\subset I_1~\textrm{and}~I_3\subset I_2^c)~&\textrm{or}~(I_3\subset I_1^c~\textrm{and}~I_3\subset I_2)~\textrm{or}\\
(I_3\subset I_1^c~\textrm{and}~I_3\subset I_2^c)~&\textrm{or}~(I_3\subset I_1^c~\textrm{and}~I_3\supset I_2).
\end{align*}
Now, inspecting each one of these four cases, it is easy to see that the intersection $D_{I_1}\cap D_{I_2}\cap D_{I_3}$ is an F-curve.
\end{proof}
\begin{lemma}
\label{normalformintersection}
Let $s_{I,J,K}$ and $s_{L,M,N}$ be two distinct boundary $2$-strata on $\overline{M}_{0,7}$ satisfying the condition $s_{I,J,K}**s_{L,M,N}$. Then we can write $\sigma_{I,J,K}\cdot\sigma_{L,M,N}=\delta_{I_1}\cdot\delta_{I_2}\cdot\delta_{I_3}\cdot\delta_{I_4}$ where, either the four boundary divisors $\delta_{I_1},\delta_{I_2},\delta_{I_3}$ and $\delta_{I_4}$ are pairwise distinct, or exactly two of them are equal. In the latter case, we assume that $I_3=I_4$. In any case, we assume that $I_1\cap I_2=\emptyset$ and $|I_1|\leq|I_2|$.
\end{lemma}
\begin{proof}
Write $\sigma_{I,J,K}\cdot\sigma_{L,M,N}=\delta_I\cdot\delta_K\cdot\delta_L\cdot\delta_N$. Obviously $\delta_I\neq\delta_K$ and $\delta_L\neq\delta_N$. If two boundary divisors among $\delta_I,\delta_K,\delta_L$ and $\delta_N$ are equal, assume without loss of generality that $\delta_K=\delta_L$. Then we must have that $\delta_N\neq\delta_I$, or we would have $s_{I,J,K}=s_{L,M,N}$. Also, $\delta_N\neq\delta_K=\delta_L$. This proves that there can be at most two boundary divisors among $\delta_I,\delta_K,\delta_L$ and $\delta_N$ that are equal. So, let us write $\delta_I\cdot\delta_K\cdot\delta_L\cdot\delta_N=\delta_A\cdot\delta_B\cdot\delta_{I_3}\cdot\delta_{I_4}$, where $\{I,K,L,N\}=\{A,B,I_3,I_4\}$ and $I_3=I_4$ in case two boundary divisors among $\delta_I,\delta_K,\delta_L$ and $\delta_N$ coincide. Finally, we can obviously rewrite $\delta_A\cdot\delta_B=\delta_{I_1}\cdot\delta_{I_2}$ with $I_1\cap I_2=\emptyset$ (here we use the hypothesis $s_{I,J,K}**s_{L,M,N}$) and $|I_1|\leq|I_2|$.
\end{proof}
\begin{proposition}
\label{intersectiondistinct}
Let $s_{I,J,K}$ and $s_{L,M,N}$ be two distinct boundary $2$-strata on $\overline{M}_{0,7}$ such that $s_{I,J,K}**s_{L,M,N}$ (otherwise, the intersection number $\sigma_{I,J,K}\cdot\sigma_{L,M,N}$ is trivially zero). Write $\sigma_{I,J,K}\cdot\sigma_{L,M,N}=\delta_{I_1}\cdot\delta_{I_2}\cdot\delta_{I_3}\cdot\delta_{I_4}$ as prescribed by Lemma~\ref{normalformintersection} (recall that in this lemma we assumed, among other things, that $|I_1|\leq|I_2|$). Then
\begin{displaymath}
\sigma_{I,J,K}\cdot\sigma_{L,M,N}=\left\{ \begin{array}{ll}
-1~&\textrm{if $\delta_{I_3}=\delta_{I_4}$, $|I_1|=2$ and $|I_2|\in\{2,4\}$}\\
1~&\textrm{if $\delta_{I_1},\delta_{I_2},\delta_{I_3}$ and $\delta_{I_4}$ are pairwise distinct}\\
0~&\textrm{otherwise}.
\end{array} \right.
\end{displaymath}
\end{proposition}
\begin{proof}
Let us make some preliminary observations. We have that
\begin{equation*}
\sigma_{I,J,K}\cdot\sigma_{L,M,N}=\delta_{I_1}\cdot\delta_{I_2}\cdot\delta_{I_3}\cdot\delta_{I_4}=\sigma_{I_1,(I_1\cup I_2)^c,I_2}\cdot\delta_{I_3}\cdot\delta_{I_4}=[s_{I_1,(I_1\cup I_2)^c,I_2}]\cdot\delta_{I_3}\cdot\delta_{I_4}.
\end{equation*}
Define $S:=s_{I_1,(I_1\cup I_2)^c,I_2}$ and let $i\colon S\hookrightarrow\overline{M}_{0,7}$ be the inclusion morphism. Using the projection formula, we obtain that
\begin{equation*}
[S]\cdot\delta_{I_3}\cdot\delta_{I_4}=i_*[S]\cdot(\delta_{I_3}\cdot\delta_{I_4})=[S]\cdot i^*(\delta_{I_3}\cdot\delta_{I_4})=i^*(\delta_{I_3}\cdot\delta_{I_4})=(i^*\delta_{I_3})\cdot(i^*\delta_{I_4}).
\end{equation*}
Now, for $j=3,4$, $i^*\delta_{I_j}=[D_{I_1}\cap D_{I_2}\cap D_{I_j}]$, where $D_{I_1}\cap D_{I_2}\cap D_{I_j}$ is an F-curve by Lemma~\ref{avolteritornano}. So $i^*\delta_{I_3}$ and $i^*\delta_{I_4}$ are two equivalence classes of F-curves on the boundary $2$-stratum $S$. There are two possibilities for $S$ up to isomorphism.
\begin{itemize}
\item[(i)] If $|I_1|=2$ and $|I_2|\in\{2,4\}$, then $S\cong\overline{M}_{0,5}$. By Kapranov's blow up construction of $\overline{M}_{0,n}$ (see \cite{kapranov}), we know that $\overline{M}_{0,5}$ is isomorphic to the blow up of $\mathbb{P}^2$ at four points in general linear position. Moreover, the F-curves of $\overline{M}_{0,5}$ correspond to the exceptional divisors of the blow up, and the strict transforms of the lines spanned by the blown up points.
\item[(ii)]  If $|I_2|=3$ and $|I_1|\in\{2,3\}$, then $S\cong\overline{M}_{0,4}\times\overline{M}_{0,4}$, which is isomorphic to $\mathbb{P}^1\times\mathbb{P}^1$. An F-curve on $S$ corresponds to a line on $\mathbb{P}^1\times\mathbb{P}^1$ in the form $\{p\}\times\mathbb{P}^1$ or $\mathbb{P}^1\times\{p\}$ for some point $p\in\mathbb{P}^1$.
\end{itemize}
Observe that in case (i) (resp. case (ii)) the self-intersection of an F-curve is $-1$ (resp. 0), and in both cases two distinct F-curves intersect at one point if and only if their intersection number is $1$.

Now, let us prove our intersection formula for $\sigma_{I,J,K}\cdot\sigma_{L,M,N}$.
\begin{itemize}
\item[$\bullet$] $|I_1|=2$ and $|I_2|=2$. Up to permuting the labels, we have that
\begin{equation*}
S\cong\overline{M}_{0,\{1,2,x\}}\times\overline{M}_{0,\{x,3,4,5,y\}}\times\overline{M}_{0,\{y,6,7\}}\cong\overline{M}_{0,\{x,3,4,5,y\}},
\end{equation*}
where $x$ and $y$ are the nodes of the stable $7$-pointed rational curve corresponding to the generic point of $S$. If $\delta_{I_3}=\delta_{I_4}$, then $(i^*\delta_{I_3})\cdot(i^*\delta_{I_4})$ is equal to the self-intersection of an F-curve on $S\cong\overline{M}_{0,\{x,3,4,5,y\}}$, which gives $\sigma_{I,J,K}\cdot\sigma_{L,M,N}=-1$. So let us assume that $\delta_{I_1},\delta_{I_2},\delta_{I_3}$ and $\delta_{I_4}$ are pairwise distinct. Given $j=3,4$, since $I_1**I_j$ and $I_2**I_j$, then $i^*\delta_{I_j}$ is equal to one of the following boundary divisors on $\overline{M}_{0,\{x,3,4,5,y\}}$
\begin{equation*}
\delta_{34},\delta_{35},\delta_{45}~\textrm{or}~\delta_{345}.
\end{equation*}
If $i^*\delta_{I_3}=\delta_{34},\delta_{35}$ or $\delta_{45}$, then $i^*\delta_{I_4}=\delta_{345}$ because $I_3**I_4$ and $\delta_{I_3}\neq\delta_{I_4}$. If $i^*\delta_{I_3}=\delta_{345}$, then $i^*\delta_{I_4}$ has to be equal to $\delta_{34},\delta_{35}$ or $\delta_{45}$. In any case, $\sigma_{I,J,K}\cdot\sigma_{L,M,N}=1$.
\item[$\bullet$] $|I_1|=2$ and $|I_2|=4$. We have isomorphisms
\begin{equation*}
S\cong\overline{M}_{0,\{1,2,x\}}\times\overline{M}_{0,\{x,3,y\}}\times\overline{M}_{0,\{y,4,5,6,7\}}\cong\overline{M}_{0,\{y,4,5,6,7\}}.
\end{equation*}
If $\delta_{I_3}=\delta_{I_4}$, then again $(i^*\delta_{I_3})\cdot(i^*\delta_{I_4})$ is equal to the self-intersection of an F-curve on $S\cong\overline{M}_{0,\{y,4,5,6,7\}}$, which gives $\sigma_{I,J,K}\cdot\sigma_{L,M,N}=-1$. Let us assume that $\delta_{I_3}\neq\delta_{I_4}$. Given $j=3,4$, then $i^*\delta_{I_j}$ is equal to one of the following boundary divisors on $\overline{M}_{0,\{y,4,5,6,7\}}$
\begin{equation*}
\delta_{45},\delta_{46},\delta_{47},\delta_{56},\delta_{57},\delta_{67},\delta_{456},\delta_{457},\delta_{467}~\textrm{or}~\delta_{567}.
\end{equation*}
If $i^*\delta_{I_3}=\delta_{45},\delta_{46},\delta_{47},\delta_{56},\delta_{57}$ or $\delta_{67}$, then assume up to a change of labels that $i^*\delta_{I_3}=\delta_{45}$. In this case, $i^*\delta_{I_4}=\delta_{67},\delta_{456}$ or $\delta_{457}$. If $i^*\delta_{I_3}=\delta_{456},\delta_{457},\delta_{467}$ or $\delta_{567}$, assume up to a change of labels that $i^*\delta_{I_3}=\delta_{456}$. Then $i^*\delta_{I_4}$ has to be equal to $\delta_{45},\delta_{46}$ or $\delta_{56}$. Each one of these choices for $i^*\delta_{I_3}$ and $i^*\delta_{I_4}$ gives $\sigma_{I,J,K}\cdot\sigma_{L,M,N}=1$.
\item[$\bullet$] $|I_1|=2$ and $|I_2|=3$. In this case we have
\begin{equation*}
S\cong\overline{M}_{0,\{1,2,x\}}\times\overline{M}_{0,\{x,3,4,y\}}\times\overline{M}_{0,\{y,5,6,7\}}\cong\overline{M}_{0,\{x,3,4,y\}}\times\overline{M}_{0,\{y,5,6,7\}}.
\end{equation*}
If $\delta_{I_3}=\delta_{I_4}$, then $(i^*\delta_{I_3})\cdot(i^*\delta_{I_4})$ is equal to the self-intersection of an F-curve on $S\cong\overline{M}_{0,\{x,3,4,y\}}\times\overline{M}_{0,\{y,5,6,7\}}$, which gives $\sigma_{I,J,K}\cdot\sigma_{L,M,N}=0$. Now consider the case $\delta_{I_3}\neq\delta_{I_4}$. For $j=3,4$, $i^*\delta_{I_j}$ is equal to the equivalence class of one of the following divisors on $\overline{M}_{0,\{x,3,4,y\}}\times\overline{M}_{0,\{y,5,6,7\}}$
\begin{equation*}
D_{34}\times\overline{M}_{0,\{y,5,6,7\}},\overline{M}_{0,\{x,3,4,y\}}\times D_{56},\overline{M}_{0,\{x,3,4,y\}}\times D_{57}~\textrm{or}~\overline{M}_{0,\{x,3,4,y\}}\times D_{67}.
\end{equation*}
Since $I_3** I_4$, the only possibility for $i^*\delta_{I_3}$ and $i^*\delta_{I_4}$ is to belong to two different rulings of $S$. It follows that $\sigma_{I,J,K}\cdot\sigma_{L,M,N}=1$.
\item[$\bullet$] $|I_1|=3$ and $|I_2|=3$. Then
\begin{equation*}
S\cong\overline{M}_{0,\{1,2,3,x\}}\times\overline{M}_{0,\{x,4,y\}}\times\overline{M}_{0,\{y,5,6,7\}}\cong\overline{M}_{0,\{1,2,3,x\}}\times\overline{M}_{0,\{y,5,6,7\}}.
\end{equation*}
If $\delta_{I_3}=\delta_{I_4}$, then $(i^*\delta_{I_3})\cdot(i^*\delta_{I_4})$ is equal to the self-intersection of an F-curve on $S\cong\overline{M}_{0,\{1,2,3,x\}}\times\overline{M}_{0,\{y,5,6,7\}}$, which gives $\sigma_{I,J,K}\cdot\sigma_{L,M,N}=0$. For the case $\delta_{I_3}\neq\delta_{I_4}$, given $j=3,4$, $i^*\delta_{I_j}$ is equal to the equivalence class of one of the following divisors on $\overline{M}_{0,\{1,2,3,x\}}\times\overline{M}_{0,\{y,5,6,7\}}$
\begin{equation*}
D_{12}\times\overline{M}_{0,\{y,5,6,7\}},D_{13}\times\overline{M}_{0,\{y,5,6,7\}},D_{23}\times\overline{M}_{0,\{y,5,6,7\}},
\end{equation*}
\begin{equation*}
\overline{M}_{0,\{1,2,3,x\}}\times D_{56},\overline{M}_{0,\{1,2,3,x\}}\times D_{57}~\textrm{or}~\overline{M}_{0,\{1,2,3,x\}}\times D_{67}.
\end{equation*}
$I_3** I_4$ implies that $i^*\delta_{I_3}$ and $i^*\delta_{I_4}$ belong to two different rulings of $S$. In particular, $\sigma_{I,J,K}\cdot\sigma_{L,M,N}=1$.
\end{itemize}
At this point, the claimed intersection formula sums up all the considerations we made so far.
\end{proof}
\subsection{Self-intersection of a boundary $2$-stratum}
We want to compute $\sigma_{I,J,K}^2=\delta_I\cdot\delta_K\cdot\delta_I\cdot\delta_K$. The idea is to find an appropriate Keel relation (see \cite[page 569, Theorem 1(2)]{keel}) that allows us to replace $\delta_I$ and reduce the calculation to the previous case.
\begin{proposition}
\label{intersectionsame}
Let $\sigma_{I,J,K}$ be the equivalence class of a boundary $2$-stratum with $|I|\leq|K|$. Then
\begin{displaymath}
\sigma_{I,J,K}^2=\left\{ \begin{array}{ll}
0~&\textrm{if $|I|=2$ and $|J|=1$}\\
2~&\textrm{if $|J|=3$}\\
1~&\textrm{otherwise}.
\end{array} \right.
\end{displaymath}
\end{proposition}
\begin{proof}
Up to relabeling the markings, it is enough to prove that $\sigma_{12,3,4567}^2=0$, $\sigma_{123,4,567}^2=\sigma_{12,34,567}^2=1$ and $\sigma_{12,345,67}^2=2$.
\begin{itemize}
\item[$\bullet$] $\sigma_{12,3,4567}^2=\delta_{12}\cdot\delta_{4567}\cdot\delta_{12}\cdot\delta_{4567}$. Let us use the boundary relation
\begin{equation*}
\sum_{\substack{1,2\in S\\3,4\in S^c}}\delta_S=\sum_{\substack{1,3\in S\\2,4\in S^c}}\delta_S\Rightarrow\delta_{12}=\delta_{13}+\delta_{135}+\delta_{136}+\delta_{137}+\delta_{1356}+\delta_{1357}+
\end{equation*}
\begin{equation*}
+\delta_{1367}+\delta_{13567}-\delta_{125}-\delta_{126}-\delta_{127}-\delta_{1256}-\delta_{1257}-\delta_{1267}-\delta_{12567}.
\end{equation*}
But now, if $\delta_T$ is one of the boundary divisors that appear in the expression we just found for $\delta_{12}$, then $\{1,2\}**T$ is false or $\{4,5,6,7\}**T$ is false. Hence, $\sigma_{12,3,4567}^2=0$.
\item[$\bullet$] $\sigma_{123,4,567}^2=\delta_{123}\cdot\delta_{567}\cdot\delta_{123}\cdot\delta_{567}$. Consider
\begin{equation*}
\sum_{\substack{1,2\in S\\4,5\in S^c}}\delta_S=\sum_{\substack{1,4\in S\\2,5\in S^c}}\delta_S\Rightarrow\delta_{123}=\delta_{14}+\delta_{143}+\delta_{146}+\delta_{147}+\delta_{1436}+\delta_{1437}+
\end{equation*}
\begin{equation*}
+\delta_{1467}+\delta_{14367}-\delta_{12}-\delta_{126}-\delta_{127}-\delta_{1236}-\delta_{1237}-\delta_{1267}-\delta_{12367}.
\end{equation*}
After replacing $\delta_{123}$ with the new expression and distributing, we get $\sigma_{123,4,567}^2=-\delta_{123}\cdot\delta_{567}\cdot\delta_{12}\cdot\delta_{567}=-\delta_{12}\cdot\delta_{4567}\cdot\delta_{567}\cdot\delta_{567}=-(-1)=1$.
\item[$\bullet$] $\sigma_{12,34,567}^2=\delta_{12}\cdot\delta_{567}\cdot\delta_{12}\cdot\delta_{567}$. We use the following boundary relation
\begin{equation*}
\sum_{\substack{1,2\in S\\3,5\in S^c}}\delta_S=\sum_{\substack{1,3\in S\\2,5\in S^c}}\delta_S\Rightarrow\delta_{12}=\delta_{13}+\delta_{134}+\delta_{136}+\delta_{137}+\delta_{1346}+\delta_{1347}+
\end{equation*}
\begin{equation*}
+\delta_{1367}+\delta_{13467}-\delta_{124}-\delta_{126}-\delta_{127}-\delta_{1246}-\delta_{1247}-\delta_{1267}-\delta_{12467}\Rightarrow
\end{equation*}
\begin{equation*}
\sigma_{12,34,567}^2=-\delta_{12}\cdot\delta_{567}\cdot\delta_{124}\cdot\delta_{567}=-\delta_{12}\cdot\delta_{3567}\cdot\delta_{567}\cdot\delta_{567}=1.
\end{equation*}
\item[$\bullet$] $\sigma_{12,345,67}^2=\delta_{12}\cdot\delta_{67}\cdot\delta_{12}\cdot\delta_{67}$.
\begin{equation*}
\sum_{\substack{1,2\in S\\3,6\in S^c}}\delta_S=\sum_{\substack{1,3\in S\\2,6\in S^c}}\delta_S\Rightarrow\delta_{12}=\delta_{13}+\delta_{134}+\delta_{135}+\delta_{137}+\delta_{1345}+\delta_{1347}+
\end{equation*}
\begin{equation*}
+\delta_{1357}+\delta_{13457}-\delta_{124}-\delta_{125}-\delta_{127}-\delta_{1245}-\delta_{1247}-\delta_{1257}-\delta_{12457}\Rightarrow
\end{equation*}
\begin{equation*}
\sigma_{12,345,67}^2=-\delta_{12}\cdot\delta_{67}\cdot\delta_{124}\cdot\delta_{67}-\delta_{12}\cdot\delta_{67}\cdot\delta_{125}\cdot\delta_{67}-\delta_{12}\cdot\delta_{67}\cdot\delta_{1245}\cdot\delta_{67}=
\end{equation*}
\begin{equation*}
=-\delta_{12}\cdot\delta_{3567}\cdot\delta_{67}\cdot\delta_{67}-\delta_{12}\cdot\delta_{3467}\cdot\delta_{67}\cdot\delta_{67}-\delta_{12}\cdot\delta_{367}\cdot\delta_{67}\cdot\delta_{67}=1+1-0=2.
\end{equation*}
\end{itemize}
\end{proof}
\begin{remark}
As one of the referees pointed out, Proposition~\ref{intersectionsame} can also be proved using \cite[Lemma 3.5]{edidin}. Say we want to compute $\sigma_{I,J,K}^2$. Then, adopting the same notation used in \cite[Lemma 3.5]{edidin}, one can take $B=s_{I,J,K}$ and let $X\rightarrow B$ be the pullback of the universal family on $\overline{M}_{0,7}$ with respect to the inclusion $s_{I,J,K}\hookrightarrow\overline{M}_{0,7}$. Then the intersection number $\sigma_{I,J,K}^2$ can be computed using the formula provided at the end of \cite[Lemma 3.5]{edidin}.
\end{remark}
\subsection{Equivalence classes of boundary $2$-strata}
\label{studynumericalequivalence}
So far, we considered set theoretically distinct boundary $2$-strata. However, we are interested in studying distinct equivalence classes of boundary $2$-strata.
\begin{proposition}
\label{equivalenceclassesmadeclear}
Consider $\sigma_{I,J,K}$ and $\sigma_{L,M,N}$ with $|I|\leq|K|$, $|L|\leq|N|$ and $s_{I,J,K}\neq s_{L,M,N}$. Then $\sigma_{I,J,K}=\sigma_{L,M,N}\Leftrightarrow I\cup J=L\cup M$ and $|I\cup J|=3$.
\end{proposition}
\begin{proof}

\

\noindent($\Leftarrow$) Assume $\{a,b,c,d,e,f,g\}=[7]$ and let $I\cup J=\{a,b,c\}$.
Consider the boundary divisor $D_{abc,defg}\cong\overline{M}_{0,4}\times\overline{M}_{0,5}\cong\mathbb{P}^1\times\overline{M}_{0,5}$. Let $\pi\colon\mathbb{P}^1\times\overline{M}_{0,5}\rightarrow\mathbb{P}^1$ be the usual projection morphism. If $C$ is the stable $7$-pointed rational curve corresponding to the generic point of $D_{abc}$, assume the node of $C$ and the labels $b,c$ fixed on the twig which contains $a,b$ and $c$. So we can think of $a$ as parametrizing $\mathbb{P}^1$, and therefore $\pi^{-1}(b)=s_{ab,c,defg}$, $\pi^{-1}(c)=s_{ac,b,defg}$. In conclusion, $s_{ab,c,defg}$ and $s_{ac,b,defg}$ are rationally equivalent.

\

\noindent($\Rightarrow$) Let us prove the contrapositive. We proceed by enumerating all the possible cases.
\begin{itemize}
\item[(i)] $|J|=3$. Then $2=\sigma_{I,J,K}\cdot\sigma_{I,J,K}\neq\sigma_{L,M,N}\cdot\sigma_{I,J,K}\in\{-1,0,1\}\Rightarrow\sigma_{I,J,K}\neq\sigma_{L,M,N}$.

\item[(ii)] $|I|=|J|=2$. Up to relabeling, we can assume that $\sigma_{I,J,K}=\sigma_{12,34,567}$. The boundary $2$-stratum $\sigma_{L,M,N}$ can be in one of the following forms
\begin{equation*}
\sigma_{ab,cd,efg},~\sigma_{abc,d,efg}~\textrm{or}~\sigma_{ab,c,defg}
\end{equation*}
($\sigma_{ab,cde,fg}$ is excluded because of what we just discussed in (i)). In any case, we can write $\sigma_{L,M,N}=\delta_S\cdot\delta_T$ with $|S|=4$. Therefore, $\sigma_{12,34,567}\cdot\sigma_{L,M,N}=\delta_{12}\cdot\delta_{1234}\cdot\delta_{S}\cdot\delta_T$ can be equal to just $0$ or $-1$ (more in detail, if $S**\{1234\}$, then $S=\{1,2,3,4\}$ and the intersection can be either $0$ or $-1$). However $\sigma_{12,34,567}\cdot\sigma_{12,34,567}=1$, so $\sigma_{12,34,567}\neq\sigma_{L,M,N}$.

\item[(iii)] $|J|=1$ and $|I|=3$. This case uses the same strategy we adopted in (ii).

\item[(iv)] $|J|=1$ and $|I|=2$. We can assume $\sigma_{I,J,K}=\sigma_{12,3,4567}$. Because of what we proved so far, we can assume that $s_{L,M,N}=s_{ab,c,defg}$. By our hypotheses, we also have that $s_{ab,c,defg}$ has to be different from $s_{13,2,4567}$ and $s_{23,1,4567}$. But now, up to permuting $\{4,5,6,7\}$ and $\{1,2\}$ (which leave $\sigma_{12,3,4567}$ unchanged), there are few possibilities for $\sigma_{ab,c,defg}$, which are
\begin{equation*}
\sigma_{12,4,3567},~\sigma_{13,4,2567},~\sigma_{14,2,3567},~\sigma_{14,3,2567},~\sigma_{14,5,2367},
\end{equation*}
\begin{equation*}
\sigma_{34,1,2567},~\sigma_{34,5,1267},~\sigma_{45,1,2367},~\sigma_{45,3,1267},~\sigma_{45,6,1237}.
\end{equation*}
In each case, one can compute that $\sigma_{ab,c,defg}\cdot\sigma_{45,123,67}=0$ using Proposition~\ref{intersectiondistinct}. But $\sigma_{12,3,4567}\cdot\sigma_{45,123,67}=1$ again by Proposition~\ref{intersectiondistinct}, and therefore $\sigma_{12,3,4567}\neq\sigma_{ab,c,defg}$.
\end{itemize}
\end{proof}
Now, an easy count tells us that there are $420$ distinct equivalence classes of boundary $2$-strata on $\overline{M}_{0,7}$. In addition, these $420$ equivalence classes generate distinct rays in $\eff_2(\overline{M}_{0,7})$ as we prove in the next proposition.
\begin{proposition}
\label{distinctrays}
Distinct equivalence classes of boundary $2$-strata on $\overline{M}_{0,7}$ generate distinct rays in the cone $\emph{Eff}_2(\overline{M}_{0,7})$.
\end{proposition}
\begin{proof}
We say that a boundary $2$-stratum $\sigma_{I,J,K}$ is of type $(a,b,c)$ if $\{a,b,c\}=\{|I|,|J|,|K|\}$. Let $\alpha$ and $\beta$ be two distinct boundary $2$-strata on $\overline{M}_{0,7}$. Assume by contradiction that we can find $r\in\mathbb{R}_{>0}$, $r\neq1$, such that $\alpha=r\beta$.

There are three cases to discuss.
\begin{itemize}
\item[$\bullet$] $\alpha$ and $\beta$ are not of type $(2,1,4)$. Then $\alpha^2=r^2\beta^2\neq0$ by Proposition~\ref{intersectionsame}, so that $r=\sqrt{\alpha^2/\beta^2}$. Considering all the possible cases for $\alpha^2$ and $\beta^2$, we see that $r\in\{1/\sqrt{2},\sqrt{2}\}$, which cannot be because $r$ has to be rational.
\item[$\bullet$] Exactly one among $\alpha$ and $\beta$ is of type $(2,1,4)$. This is impossible because one side of the equality $\alpha^2=r^2\beta^2$ would be zero and the other not.
\item[$\bullet$] Both $\alpha$ and $\beta$ are of type $(2,1,4)$. Since $\alpha\neq0$, we can find a boundary $2$-stratum $\gamma$ such that $\alpha\cdot\gamma\neq0$. According to Propositions~\ref{intersectiondistinct} and \ref{intersectionsame}, we must have that $|\alpha\cdot\gamma|=1$ and $|\beta\cdot\gamma|\in\{0,1\}$. In any case, the equality $|\alpha\cdot\gamma|=r|\beta\cdot\gamma|$ gives a contradiction.
\end{itemize}
\end{proof}
Recent work of Chen and Coskun (see \cite{chencoskun}) shows that the $420$ equivalence classes of boundary $2$-strata on $\overline{M}_{0,7}$ generate extremal rays of $\eff_2(\overline{M}_{0,7})$.
\begin{theorem}[{{\cite[Theorem 6.1]{chencoskun}}}]
Equivalence classes of boundary strata of codimension $2$ on $\overline{M}_{0,n}$ are extremal in $\emph{Eff}^2(\overline{M}_{0,n})$.
\end{theorem}
To conclude, the next corollary completely describes the cone $V_2(\overline{M}_{0,7})$ and sums up what we know about $\eff_2(\overline{M}_{0,7})$ so far.
\begin{corollary}
\label{IseeV2}
The cone $\emph{Eff}_2(\overline{M}_{0,7})$ has at least $420$ extremal rays, which are generated by the distinct equivalence classes of the boundary $2$-strata on $\overline{M}_{0,7}$. In particular, the closed cone $V_2(\overline{M}_{0,7})$ has exactly $420$ extremal rays.
\end{corollary}
\subsection{The intersection form $N_2(\overline{M}_{0,7})\times N_2(\overline{M}_{0,7})\rightarrow\mathbb{R}$}
\label{studyofthebilinearform}
The real vector space $N_2(\overline{M}_{0,7})$ is equipped with a symmetric bilinear form $Q\colon N_2(\overline{M}_{0,7})\times N_2(\overline{M}_{0,7})\rightarrow\mathbb{R}$ given by the intersection between equivalence classes of $2$-cycles. Since $Q$ is nondegenerate, then $Q$ has rank equal to $\dim_{\mathbb{R}}N_2(\overline{M}_{0,7})$.
\begin{proposition}
\label{dimension127}
$\dim_{\mathbb{R}}N_2(\overline{M}_{0,7})=127$.
\end{proposition}
\begin{proof}
Let $\mathbb{K}$ be our base field. We know that the equivalence classes of the boundary $2$-strata span $N_2(\overline{M}_{0,7})$ in any characteristic. Moreover, the linear dependence relations between the equivalence classes of the boundary $2$-strata on $\overline{M}_{0,7}$, only depend on the combinatorics of the intersection between the boundary $2$-strata (that we just studied in Proposition~\ref{intersectiondistinct} and Proposition~\ref{intersectionsame}), and all this does not depend on $\textrm{char}(\mathbb{K})$. Hence, $\dim_{\mathbb{R}}N_2(\overline{M}_{0,7})$ does not depend on the characteristic, and we can assume that $\mathbb{K}=\mathbb{C}$.

As a complex variety, $\overline{M}_{0,7}$ is an HI scheme. An HI scheme $X$ is a scheme of characteristic zero such that the canonical map $A_*(X)\rightarrow H_*(X;\mathbb{Z})$ from the Chow groups to the homology is an isomorphism (see \cite[Appendix]{keel} for more details). It follows that the Chow group $\textrm{CH}_2(\overline{M}_{0,7})$ is isomorphic to the homology group $H_4(\overline{M}_{0,7};\mathbb{Z})$. Therefore, the dimension of $N_2(\overline{M}_{0,7})\cong\textrm{CH}_2(\overline{M}_{0,7})\otimes_{\mathbb{Z}}\mathbb{R}$ as a real vector space is equal to $b_4$, the $4$-th Betti number of $\overline{M}_{0,7}$.

We can find $b_4$ by computing $P_{\overline{M}_{0,7}}(q)=\sum_{j\geq0}b_jq^j$, the Poincar\'e polynomial of $\overline{M}_{0,7}$. We compute this polynomial by using a recursive formula in \cite[Section 5]{chengibneykrashen}, which gives the Poincar\'e polynomial of the space $T_{d,n}$, the compact moduli space of stable $n$-pointed rooted trees of $d$-dimensional projective spaces. In our case, $\overline{M}_{0,7}=T_{1,6}$ (see \cite[Proposition 3.4.3]{chengibneykrashen}), and one can compute that $P_{\overline{M}_{0,7}}(q)=1+42q^2+127q^4+42q^6+q^8$.
\end{proof}
\begin{proposition}
\label{signature8641}
The bilinear form $Q$ has signature $(86,41)$.
\end{proposition}
\begin{proof}
The $420$ equivalence classes of the boundary $2$-strata span $N_2(\overline{M}_{0,7})$. Therefore, we can choose $127$ of these $2$-cycles to form a basis for $N_2(\overline{M}_{0,7})$, and a matrix representation for $Q$ is given by the intersection matrix of these $127$ equivalence classes of boundary $2$-strata. Since this matrix just depends on the combinatorics of the intersection between the boundary $2$-strata, we have that the signature of $Q$ is independent of the characteristic of the base field. So let $\mathbb{C}$ be our base field.

With this assumption, we have that $\overline{M}_{0,7}$ is an HI scheme and a smooth manifold, implying that $N_2(\overline{M}_{0,7})\cong H^4(\overline{M}_{0,7};\mathbb{R})$. Using the Hodge-Riemann bilinear relations (see \cite[Chapter 0]{griffithsharris}), one has that
\begin{equation*}
I(\overline{M}_{0,7})=\sum_{p+q~\textrm{is even}}(-1)^ph^{p,q},
\end{equation*}
where $I(\overline{M}_{0,7})$ is the index of $\overline{M}_{0,7}$ (i.e. the number of positive eigenvalues minus the number of negative eigenvalues in a matrix representation of $Q$), and $h^{p,q}$ the Hodge numbers of $\overline{M}_{0,7}$.

Now, knowing that the Poincar\'e polynomial of $\overline{M}_{0,7}$ is $P_{\overline{M}_{0,7}}(q)=1+42q^2+127q^4+42q^6+q^8$ (see the proof of Proposition~\ref{dimension127}), and using the Hodge decomposition, we can compute that
\begin{equation*}
I(\overline{M}_{0,7})=2h^{0,0}+4h^{2,0}-2h^{1,1}+2h^{4,0}-2h^{3,1}+h^{2,2}=2+0-84+0-0+127=45,
\end{equation*}
implying that the signature of $Q$ is $(86,41)$.
\end{proof}
Under a more arithmetic perspective, we can view $Q$ as a bilinear form on $H_4(\overline{M}_{0,7};\mathbb{Z})$ (which is torsion-free). In this case, $Q$ is unimodular by Poincar\'e duality and odd by Proposition~\ref{intersectionsame}.


\section{Lift of effective cycles}
\label{lift!}
The technique we are about to describe allows to construct an effective $k$-cycle on $\overline{M}_{0,n+1}$ given an effective $k$-cycle on $\overline{M}_{0,n}$.
\subsection{}
Let $\pi\colon\overline{M}_{0,n+1}\rightarrow\overline{M}_{0,n}$ be the map forgetting the $(n+1)$-th label. Consider the boundary divisor $D_{n,n+1}$ and let $i\colon D_{n,n+1}\hookrightarrow\overline{M}_{0,n+1}$ be the inclusion morphism. The following varieties can be naturally identified
\begin{equation*}
\overline{M}_{0,n}\equiv\overline{M}_{0,[n-1]\cup\{x\}}\times\overline{M}_{0,\{n,n+1,x\}}\equiv D_{n,n+1},
\end{equation*}
and therefore we have a commutative diagram
\begin{center}
\begin{tikzpicture}[>=angle 90]
\matrix(a)[matrix of math nodes,
row sep=2em, column sep=2em,
text height=1.5ex, text depth=0.25ex]
{\overline{M}_{0,n}&\overline{M}_{0,n+1}\\
&\overline{M}_{0,n}.\\};
\path[->] (a-1-1) edge node[above]{$i$}(a-1-2);
\path[->] (a-1-1) edge node[below left]{$\textrm{id}$}(a-2-2);
\path[->] (a-1-2) edge node[right]{$\pi$}(a-2-2);
\end{tikzpicture}
\end{center}
\begin{definition}
If $\alpha\in\textrm{Eff}_k(\overline{M}_{0,n})$, then $i_*\alpha\in\textrm{Eff}_k(\overline{M}_{0,n+1})$ will be called the \emph{lift of $\alpha$ to $\overline{M}_{0,n+1}$}.
\end{definition}
Observe that, instead of just considering $D_{n,n+1}$, one can do a similar construction with any $D_{ab}$, $\{a,b\}\subset[n+1]$. As the following lemma explains, some of the properties of $\alpha$ are preserved after we lift it.
\begin{liftinglemma}
Let $k$ and $n$ be integers such that $0<k<n-3$. Let $\alpha$ be the equivalence class of an effective $k$-cycle on $\overline{M}_{0,n}$. Consider the maps $i\colon\overline{M}_{0,n}\rightarrow\overline{M}_{0,n+1}$ and $\pi\colon\overline{M}_{0,n+1}\rightarrow\overline{M}_{0,n}$ as above. Then
\begin{itemize}
\item[\emph{(i)}] if $\alpha\in\emph{Eff}_k(\overline{M}_{0,n})\setminus V_k(\overline{M}_{0,n})$, then $i_*\alpha\in\emph{Eff}_k(\overline{M}_{0,n+1})\setminus V_k(\overline{M}_{0,n+1})$;
\item[\emph{(ii)}] if $\alpha$ is extremal in $\emph{Eff}_k(\overline{M}_{0,n})$, then $i_*\alpha$ is extremal in $\emph{Eff}_k(\overline{M}_{0,n+1})$.
\end{itemize}
\end{liftinglemma}
\begin{proof}

\

\begin{itemize}
\item[(i)] Assume by contradiction that $i_*\alpha\in V_k(\overline{M}_{0,n+1})$. Therefore we can write $i_*\alpha=\sum_{j=1}^mr_j[Z_j]$, where $r_j\in\mathbb{R}_{>0}$ and $Z_j\subset\overline{M}_{0,n+1}$ are boundary $k$-strata. But then $\alpha=\textrm{id}_*\alpha=\pi_*i_*\alpha=\sum_{j=1}^mr_j\pi_*[Z_j]\in V_k(\overline{M}_{0,n})$, because $\pi_*[Z_j]$ is either zero or the equivalence class of a boundary $k$-stratum on $\overline{M}_{0,n}$ for all $j$. This is a contradiction.
\item[(ii)] Assume that $i_*\alpha=\sum_{j=1}^mr_j[Z_j]$, where $r_j\in\mathbb{R}_{>0}$ and $Z_j\subset\overline{M}_{0,n+1}$ are irreducible and effective $k$-cycles. We prove that $[Z_j]$ is proportional to $i_*\alpha$ for all $j=1,\ldots,m$.

Consider the reduction morphism $f_{\mathcal{A}}\colon\overline{M}_{0,n+1}\rightarrow\overline{M}_{0,\mathcal{A}}$ where $\mathcal{A}$ is the weight data $(\frac{1}{n-1},\ldots,\frac{1}{n-1},1,1)$ (see \cite{hassett}). The exceptional locus of $f_{\mathcal{A}}$ is exactly $D_{n,n+1}=i(\overline{M}_{0,n})$ and $f_{\mathcal{A}}(D_{n,n+1})$ is a point. In particular $f_{\mathcal{A}*}i_*\alpha=0$, implying that
\begin{equation*}
\sum_{j=1}^mr_jf_{\mathcal{A}*}[Z_j]=0.
\end{equation*}
Since $\overline{M}_{0,\mathcal{A}}$ is projective, we have that $f_{\mathcal{A}*}[Z_j]=0$ for all $j$, which is equivalent to $\dim(f_{\mathcal{A}}(Z_j))<k$ for all $j$. This implies that, given any $j$, $Z_j\subset D_{n,n+1}$. Define $Z_j'=i^{-1}Z_j$, so that $i_*[Z_j']=Z_j$, and therefore $\sum_{j=1}^mr_j[Z_j']$ is an effective $k$-cycle on $\overline{M}_{0,n}$ such that
\begin{equation*}
i_*\sum_{j=1}^mr_j[Z_j']=\sum_{j=1}^mr_j[Z_j]=i_*\alpha.
\end{equation*}
The pushforward morphism $i_*$ is injective on $k$-cycles, because $\pi_*\circ i_*$ is the identity. It follows that $\alpha=\sum_{j=1}^mr_j[Z_j']$, and hence each $[Z_j']$ is proportional to $\alpha$ by the extremality of $\alpha$ in $\eff_k(\overline{M}_{0,n})$. In particular, each $[Z_j]$ has to be proportional to $i_*\alpha$ for all $j$.
\end{itemize}
\end{proof}
\noindent Alternatively, the following proposition can be used to prove the second part of the lifting lemma.
\begin{proposition}[{{\cite[Proposition 2.5]{chencoskun}}}]
\label{usethisforextremality}
Let $\gamma\colon Y\rightarrow X$ be a morphism between two projective varieties. Assume that $A_k(Y)\rightarrow N_k(Y)$ is an isomorphism and that the composite $\gamma_*\colon A_k(Y)\rightarrow A_k(X)\rightarrow N_k(X)$ is injective. Moreover, assume that $f\colon X\rightarrow W$ is a morphism to a projective variety $W$ whose exceptional locus is contained in $\gamma(Y)$. If a $k$-dimensional subvariety $Z\subset Y$ is an extremal cycle in $\emph{Eff}_k(Y)$ and if $\dim(\gamma(Z))-\dim(f(\gamma(Z)))>0$, then $\gamma(Z)$ is also extremal in $\emph{Eff}_k(X)$.
\end{proposition}
\noindent Given this result, one can prove that $i_*\alpha$ is extremal by taking $Y=\overline{M}_{0,n}$, $X=\overline{M}_{0,n+1}$, $\gamma=i$, $W=\overline{M}_{0,\mathcal{A}}$ with $\mathcal{A}=(\frac{1}{n-1},\ldots,\frac{1}{n-1},1,1)$, $f=f_{\mathcal{A}}$ and $[Z]=\alpha$.
\subsection{Fulton's question}
\label{fulton'squestion}
The following question is attributed to Fulton.
\begin{fulton'squestion}[{{\cite[Question 1.1]{keelmckernan}}}]
Let $0< k< n-3$. Is it true that $V_k(\overline{M}_{0,n})=\eff_k(\overline{M}_{0,n})$?
\end{fulton'squestion}
\noindent Following \cite{gibneykeelmorrison} and \cite{vermeire} notation, denote the previous question with $F_k(0,n)$ (observe that the analogue question for $k=0$ or $k=n-3$ is trivial). $F_1(0,5)$ is answered positively because $\overline{M}_{0,5}$ is a del Pezzo of degree $5$. The answer to $F_1(0,6)$ and $F_1(0,7)$ is also yes, but this is a deep result of Keel and McKernan (see \cite{keelmckernan}). $F_1(0,n)$ for $n>7$ is an open question, and the conjecture that says $F_1(0,n)$ has a positive answer for $n>7$ is called the F-conjecture.

Keel and Vermeire showed that $F_{n-4}(0,n)$ has a negative answer for all $n\geq6$ (see \cite{gibneykeelmorrison}, \cite{vermeire}). This result, combined with the lifting lemma, clearly shows what is the answer to $F_k(0,n)$ for $1<k<n-4$.
\begin{corollary}
\label{intermediatecasesout}
If $1<k<n-4$, then $F_k(0,n)$ has a negative answer, or in other words $V_k(\overline{M}_{0,n})\subsetneq\emph{Eff}_k(\overline{M}_{0,n})$.
\end{corollary}
\begin{proof}
We know that $V_k(\overline{M}_{0,k+1})\subsetneq\eff_k(\overline{M}_{0,k+1})$ from \cite{gibneykeelmorrison}, \cite{vermeire}. Therefore, using the lifting lemma, we see that $V_k(\overline{M}_{0,k+2})\subsetneq\eff_k(\overline{M}_{0,k+2})$. Now, by iterating this argument, we obtain that $V_k(\overline{M}_{0,n})\subsetneq\eff_k(\overline{M}_{0,n})$.
\end{proof}


\section{Lifts to $\overline{M}_{0,7}$ of the Keel-Vermeire divisors on $\overline{M}_{0,6}$}
\label{liftsofthekeelvermeiredivisors}
The following description of the Keel-Vermeire divisors on $\overline{M}_{0,6}$ is convenient for us.
\begin{definition}
Assume $[6]=\{i,j,k,\ell,m,q\}$. A divisor on $\overline{M}_{0,6}$ in the from
\begin{equation*}
\delta_{mq,ij}^{KV}:=\delta_{im}+\delta_{jm}+\delta_{kq}+\delta_{\ell q}+2\delta_{ijm}-\delta_{mq},
\end{equation*}
is called a \emph{Keel-Vermeire divisor on $\overline{M}_{0,6}$}.
\end{definition}
\begin{properties of the keel-vermeire divisors on M06bar}
A Keel-Vermeire divisor $\delta_{mq,ij}^{KV}$ on $\overline{M}_{0,6}$ is effective and cannot be written as an effective sum of boundary divisors. This is proved in \cite{vermeire} in characteristic zero, and one can see that it actually holds in any characteristic. It is also important for us to know that the Keel-Vermeire divisors on $\overline{M}_{0,6}$ are extremal in the cone $\eff_2(\overline{M}_{0,6})$ in any characteristic. A proof of this can be found in \cite{castravettevelev13}. Observe that $\delta_{mq,ij}^{KV}=\delta_{ij,mq}^{KV}=\delta_{qm,ij}^{KV}=\delta_{mq,ji}^{KV}=\delta_{mq,k\ell}^{KV}$, therefore there are $15$ Keel-Vermeire divisors on $\overline{M}_{0,6}$.

One more property (but we do not use it in this paper) is that the Keel-Vermeire divisors together with the boundary divisors on $\overline{M}_{0,6}$ generate the cone $\eff_2(\overline{M}_{0,6})$. This was first proved by Hassett and Tschinkel in \cite{hassetttschinkel}. An alternative proof of this fact can be found in \cite{castravet} (actually, in \cite{castravet} it is proved that the Cox ring of $\overline{M}_{0,6}$ is generated by the sections of these divisors, which is a stronger condition).
\end{properties of the keel-vermeire divisors on M06bar}
Now we want to lift the Keel-Vermeire divisors to $\overline{M}_{0,7}$ and give a combinatorial description of these lifts.
\begin{proposition}
Let $[7]=\{a,b,i,j,k,\ell,m\}$. Then any lift to $\overline{M}_{0,7}$ of a Keel-Vermeire divisor on $\overline{M}_{0,6}$ can be written as the following linear combination of boundary $2$-strata
\begin{equation*}
\sigma_{im,jk\ell,ab}+\sigma_{jm,ik\ell,ab}+\sigma_{ij\ell m,k,ab}+\sigma_{ijkm,\ell,ab}+2\sigma_{ijm,k\ell,ab}-\sigma_{ijk\ell,m,ab}.
\end{equation*}
\end{proposition}
\begin{proof}
Let us choose a boundary divisor $D_{ab}$ on $\overline{M}_{0,7}$. This can be identified with $\overline{M}_{0,([7]\cup\{x\})\setminus\{a,b\}}$, where $x$ is an extra label. So, if we write $([7]\cup\{x\})\setminus\{a,b\}=\{i,j,k,l,m,x\}$, a Keel-Vermeire divisor on $\overline{M}_{0,([7]\cup\{x\})\setminus\{a,b\}}$ is in the form
\begin{equation*}
\delta_{mx,ij}^{KV}=\delta_{im}+\delta_{jm}+\delta_{kx}+\delta_{\ell x}+2\delta_{ijm}-\delta_{mx}.
\end{equation*}
If $\iota\colon D_{ab}\hookrightarrow\overline{M}_{0,7}$ is the natural inclusion, then the lift to $\overline{M}_{0,7}$ of $\delta_{mx,ij}^{KV}$ is by definition
\begin{equation*}
\iota_*\delta_{mx,ij}^{KV}=\iota_*\delta_{im}+\iota_*\delta_{jm}+\iota_*\delta_{kx}+\iota_*\delta_{\ell x}+2\iota_*\delta_{ijm}-\iota_*\delta_{mx}.
\end{equation*}
Now, each one of the pushforwards appearing in the right hand side of the previous identity, can be computed by attaching along $x$ a rational tail with the labels $\{x,a,b\}$. By doing so, we obtain the claimed $2$-cycle on $\overline{M}_{0,7}$.
\end{proof}
\begin{notation}
We use $\sigma_{ab,m,ij}^{KV}$ to denote the following lift to $\overline{M}_{0,7}$ of a Keel-Vermeire divisor on $\overline{M}_{0,6}$
\begin{equation*}
\sigma_{ab,m,ij}^{KV}:=\sigma_{im,jk\ell,ab}+\sigma_{jm,ik\ell,ab}+\sigma_{ij\ell m,k,ab}+\sigma_{ijkm,\ell,ab}+2\sigma_{ijm,k\ell,ab}-\sigma_{ijk\ell,m,ab}.
\end{equation*}
\end{notation}
The next lemma will be used several times.
\begin{lemma}
\label{risolutore}
Let $\pi_y\colon\overline{M}_{0,7}\rightarrow\overline{M}_{0,[7]\setminus\{y\}}$ be the map forgetting the label $y\in[7]$, and let $\sigma_{ab,m,ij}^{KV}$ be a lift to $\overline{M}_{0,7}$ of a Keel-Vermeire divisor on $\overline{M}_{0,6}$. Then
\begin{displaymath}
\pi_{y*}\sigma_{ab,m,ij}^{KV}=\left\{ \begin{array}{ll}
\delta_{mb,ij}^{KV}~&\textrm{if $y=a$}\\
\delta_{ma,ij}^{KV}~&\textrm{if $y=b$}\\
\delta_{ab}~&\textrm{otherwise}.
\end{array} \right.
\end{displaymath}
\end{lemma}
\begin{proof}
First observe that
\begin{equation*}
\pi_{a*}\sigma_{ab,m,ij}^{KV}=\pi_{a*}(\sigma_{im,jk\ell,ab}+\sigma_{jm,ik\ell,ab}+\sigma_{ij\ell m,k,ab}+\sigma_{ijkm,\ell,ab}+2\sigma_{ijm,k\ell,ab}-\sigma_{ijk\ell,m,ab})=
\end{equation*}
\begin{equation*}
\delta_{im}+\delta_{jm}+\delta_{kb}+\delta_{\ell b}+2\delta_{ijm}-\delta_{mb}=\delta_{mb,ij}^{KV}.
\end{equation*}
In the same way, one can prove that $\pi_{b*}\sigma_{ab,m,ij}^{KV}=\delta_{ma,ij}^{KV}$.

Let $y\in[7]\setminus\{a,b\}$. Up to relabeling, we can assume that $\sigma_{ab,m,ij}^{KV}=\sigma_{67,5,12}^{KV}$. Moreover, by the symmetries of the Keel-Vermeire divisors, we just have to prove our claim when $y=5$ or $y=1$. In the former case,
\begin{equation*}
\pi_{5*}\sigma_{67,5,12}^{KV}=\pi_{5*}(\sigma_{15,234,67}+\sigma_{25,134,67}+\sigma_{1245,3,67}+\sigma_{1235,4,67}+2\sigma_{125,34,67}-\sigma_{1234,5,67})=
\end{equation*}
\begin{equation*}
\delta_{67}+\delta_{67}+0+0+0-\delta_{67}=\delta_{67}.
\end{equation*}
Finally, if $y=1$, we have that
\begin{equation*}
\pi_{1*}\sigma_{67,5,12}^{KV}=\delta_{67}+0+0+0+0-0=\delta_{67}.
\end{equation*}
\end{proof}
Since we have $\binom{7}{2}$ choices for $D_{ab}$ and $15$ choices for a Keel-Vermeire divisor inside $D_{ab}$, in total we have $315$ lifts of Keel-Vermeire divisors to $\overline{M}_{0,7}$. The question now is whether these $315$ equivalence classes generate different extremal rays of $\eff_2(\overline{M}_{0,7})$. This is what we are about to prove.
\begin{proposition}
\label{315extremalrays}
The $315$ lifts to $\overline{M}_{0,7}$ of the Keel-Vermeire divisors on $\overline{M}_{0,6}$ generate distinct extremal rays of $\emph{Eff}_2(\overline{M}_{0,7})$ that lie outside of $V_2(\overline{M}_{0,7})$.
\end{proposition}
\begin{proof}
The extremality of these rays and the fact that they lie outside of the cone $V_2(\overline{M}_{0,7})$ follow from our lifting lemma in Section~\ref{lift!}. Let us prove that these rays are all distinct.

Consider two lifts of Keel-Vermeire divisors in the form $\sigma_{ab,m,ij}^{KV}$, $\sigma_{ab,m',i'j'}^{KV}$. Assume that $\sigma_{ab,m,ij}^{KV}=r\sigma_{ab,m',i'j'}^{KV}$ for some $r\in\mathbb{R}_{>0}$. Then we must have $\delta_{mb,ij}^{KV}=\pi_{a*}\sigma_{ab,m,ij}^{KV}=\pi_{a*}r\sigma_{ab,m',i'j'}^{KV}=r\delta_{m'b,i'j'}^{KV}$, which implies that $\delta_{mb,ij}^{KV}=\delta_{m'b,i'j'}^{KV}$ because different Keel-Vermeire divisors generate different rays. In particular, $\sigma_{ab,m,ij}^{KV}=\sigma_{ab,m',i'j'}^{KV}$. From this we conclude that lifts of different Keel-Vermeire divisors which are contained in the same boundary divisor give rise to distinct rays of $\eff_2(\overline{M}_{0,7})$.

Let us consider two distinct boundary divisors $D_{ab}$ and $D_{cd}$. Consider two lifts $\sigma_{ab,m,ij}^{KV}$ and $\sigma_{cd,m',i'j'}^{KV}$. Assume by contradiction that $\sigma_{ab,m,ij}^{KV}=r\sigma_{cd,m',i'j'}^{KV}$ for some $r\in\mathbb{R}_{>0}$. Since $D_{ab}$ and $D_{cd}$ are distinct, we can assume without loss of generality that $a\notin\{c,d\}$. It follows that
\begin{equation*}
\delta_{mb,ij}^{KV}=\pi_{a*}\sigma_{ab,m,ij}^{KV}=\pi_{a*}r\sigma_{cd,m',i'j'}^{KV}=r\delta_{cd},
\end{equation*}
which is a contradiction because a Keel-Vermeire divisor cannot be proportional to a boundary divisor.
\end{proof}
\begin{definition}
Define $V_2^{KV}(\overline{M}_{0,7})\subseteq\eff_2(\overline{M}_{0,7})$ to be the cone generated by the boundary $2$-strata on $\overline{M}_{0,7}$ and the lifts to $\overline{M}_{0,7}$ of the Keel-Vermeire divisors on $\overline{M}_{0,6}$.
\end{definition}
\begin{corollary}
The cone $\emph{Eff}_2(\overline{M}_{0,7})$ has at least $735$ extremal rays: $420$ are generated by the boundary $2$-strata and $315$ are generated by the lifts of Keel-Vermeire divisors. In particular, the closed cone $V_2^{KV}(\overline{M}_{0,7})$ has exactly $735$ extremal rays.
\end{corollary}
Now our goal is to describe $\eff_2(\overline{M}_{0,7})$ outside of the cone $V_2^{KV}(\overline{M}_{0,7})$. The first question that one may ask is whether or not $V_2^{KV}(\overline{M}_{0,7})$ is equal to $\eff_2(\overline{M}_{0,7})$. In what follows, we establish that these two cones are not equal.


\section{Embedded blow ups of $\mathbb{P}^2$ in $\overline{M}_{0,n}$}
\label{embeddedblowupsofp2inm0nbar}
\subsection{The blow up construction}
In \cite{castravettevelev12}, Castravet and Tevelev give a way to embed $\textrm{Bl}(\mathbb{P}^2)$ in $\overline{M}_{0,n}$, where the embedding and the blow up depend on the choice of $n$ points in $\mathbb{P}^2$. Moreover, they tell us how the boundary divisors pullback under this embedding. Here is their construction.
\begin{theorem}[{{\cite[Theorem 3.1]{castravettevelev12}}}]
\label{construct2cycles}
Suppose $p_1,\ldots,p_n\in\mathbb{P}^2$ are distinct points, and let $U\subset\mathbb{P}^2$ be the complement of the union of the lines spanned by these points. Consider the morphism
\begin{equation*}
F\colon U\rightarrow M_{0,n}
\end{equation*}
defined as follows: given $p\in U$, let $F(p)=[(\mathbb{P}^1;\varphi_p(p_1),\ldots,\varphi_p(p_n))]$, where $\varphi_p\colon\mathbb{P}^2\dashrightarrow\mathbb{P}^1$ is the projection from $p$. Then $F$ extends to a morphism
\begin{equation*}
F\colon\emph{Bl}_{p_1,\ldots,p_n}\mathbb{P}^2\rightarrow\overline{M}_{0,n}.
\end{equation*}
If the points $p_1,\ldots,p_n$ do not lie on a (possibly reducible) conic, then $F$ is a closed embedding. In this case the boundary divisors $\delta_I$ of $\overline{M}_{0,n}$ pullback as follows: for each line $L$ in our line arrangement, if $I\subseteq[n]$ is such that $p_i\in L\Leftrightarrow i\in I$, then $F^*\delta_I=\widehat{L}_I$ (the strict transform of $L_I$) and (assuming $|I|\geq3$) $F^*\delta_{I\setminus\{k\}}=E_k$, where $k\in I$ and $E_k$ is the exceptional divisor over $p_k$. Other boundary divisors pullback trivially.
\end{theorem}
In \cite{castravettevelev12}, this theorem is used to embed curves in $\overline{M}_{0,n}$ that are possible candidate to be counterexamples to the F-conjecture (later on in the paper, they show that these curves actually are not counterexamples by means of the ``arithmetic break" technique).

From now on, our attention is focused on this kind of embedded surfaces. Let us give a name to them.
\begin{definition}
Consider $n$ points $p_1,\ldots,p_n\in\mathbb{P}^2$ that do not lie on a (possibly reducible) conic. Then, using Theorem~\ref{construct2cycles}, the embedded surface $F\colon\textrm{Bl}_{p_1,\ldots,p_n}\hookrightarrow\overline{M}_{0,n}$ will be called an \emph{embedded blow up of $\mathbb{P}^2$ in $\overline{M}_{0,n}$}. The points $p_1,\ldots,p_n$ will be called the \emph{points associated to the embedded blow up}.
\end{definition}
\begin{remark}
If $\sigma$ is the equivalence class of an embedded blow up of $\mathbb{P}^2$ in $\overline{M}_{0,n}$, observe that the intersection properties of $\sigma$ can be studied using the projection formula. Let $\sigma_{I,J,K}$ be a codimension $2$ boundary stratum on $\overline{M}_{0,n}$. Then
\begin{equation*}
\sigma\cdot\sigma_{I,J,K}=F_*[\textrm{Bl}_{p_1,\ldots,p_n}(\mathbb{P}^2)]\cdot(\delta_I\cdot\delta_K)=[\textrm{Bl}_{p_1,\ldots,p_n}(\mathbb{P}^2)]\cdot F^*(\delta_I\cdot\delta_K)=(F^*\delta_I)\cdot(F^*\delta_K).
\end{equation*}
Now, the intersection $(F^*\delta_I)\cdot(F^*\delta_K)$ is easy to compute because the two divisors $F^*\delta_I$ and $F^*\delta_K$ can be either zero, an exceptional divisor, or the strict transform of a line that is spanned by the $n$ points in $\mathbb{P}^2$.
\end{remark}
It will be crucial to know that the Keel-Vermeire divisors on $\overline{M}_{0,6}$ can be realized as particular embedded blow ups of $\mathbb{P}^2$. This is proved by Castravet and Tevelev in \cite[Section 9]{castravettevelev13} using \emph{irreducible hypertrees}.
\begin{definition}
Let $n\geq3$ and $d\geq1$. A \emph{hypertree} $\Gamma=\{\Gamma_1,\ldots,\Gamma_d\}$ on the set $[n]$ is a collection of subsets of $[n]$ such that the following conditions are satisfied:
\begin{itemize}
\item[$\bullet$] Any subset $\Gamma_j$ has at least three elements;
\item[$\bullet$] Any $i\in[n]$ is contained in at least two subsets $\Gamma_j$;
\item[$\bullet$] (\emph{convexity axiom})
\begin{equation*}
\bigg|\bigcup_{j\in S}\Gamma_j\bigg|\geq\sum_{j\in S}(|\Gamma_j|-2)~\textrm{for any $S\subsetneq[d]$, $|S|>1$};
\end{equation*}
\item[$\bullet$] (\emph{normalization})
\begin{equation*}
n-2=\sum_{j\in[d]}(|\Gamma_j|-2).
\end{equation*}
\end{itemize}
A hypertree is \emph{irreducible} if all the inequalities in the convexity axiom are strict. A \emph{planar realization} for a hypertree $\Gamma$ on the set $[n]$ is a configuration of different points $p_1,\ldots,p_n\in\mathbb{P}^2$ such that, for any subset $S\subsetneq[n]$ with at least three points, $\{p_i\}_{i\in S}$ are collinear if and only if $S\subseteq\Gamma_j$ for some $j$.
\end{definition}
\begin{remark}
\label{kvasembeddedblowups}
It turns out that, up to a change of labels, there is a unique irreducible hypertree on the set $[6]$, and a planar realization for this is given by the intersection points of four lines in $\mathbb{P}^2$ in general linear position. In \cite[Section 9]{castravettevelev13} is proved that the embedding in $\overline{M}_{0,6}$ of the blow up of $\mathbb{P}^2$ at the six points of this planar realization gives a Keel-Vermeire divisor. Moreover, we can actually obtain all the $15$ Keel-Vermeire divisors by labeling the six points appropriately.
\end{remark}
The reason why we are interested in these embedded blow ups of $\mathbb{P}^2$ in $\overline{M}_{0,7}$ is because they allow us to provide examples of effective $2$-cycles whose equivalence classes do not lie in the cone $V_2^{KV}(\overline{M}_{0,7})$. The examples we discuss are given by what we call \emph{special hypertree surfaces}, which are related to Castravet and Tevelev irreducible hypertrees.
\subsection{Special hypertree surfaces on $\overline{M}_{0,7}$}
\label{specialhypertreesurfaces}
\begin{definition}
\label{definitionofspecialhypertreesurface}
An embedded blow up of $\mathbb{P}^2$ in $\overline{M}_{0,7}$ with associated points $p_1,\ldots,p_7$ will be called a \emph{hypertree surface on $\overline{M}_{0,7}$} if there exists $y\in[7]$ such that $p_1,\ldots,\widehat{p}_y,\ldots,p_7$ is a planar realization for an irreducible hypertree on the set $[7]\setminus\{y\}$. A hypertree surface will be called \emph{special} if we can find three distinct such $y\in[7]$.
\end{definition}
\begin{lemma}
\label{tevelevmihadetto}
Let $h\in\emph{Eff}_2(\overline{M}_{0,7})$ be the equivalence class of a hypertree surface on $\overline{M}_{0,7}$. Then $h\notin V_2(\overline{M}_{0,7})$.
\end{lemma}
\begin{proof}
Let $p_1,\ldots,p_7\in\mathbb{P}^2$ be the points associated to $h$ and assume without loss of generality that the points $p_1,\ldots,p_6$ form a planar realization for an irreducible hypertree on the set $[6]$. Arguing by contradiction, let $h=\sum \alpha_{I,J,K}\sigma_{I,J,K}$ for some coefficients $\alpha_{I,J,K}\in\mathbb{R}_{\geq0}$. If $\pi_{7}\colon\overline{M}_{0,7}\rightarrow\overline{M}_{0,6}$ is the morphism forgetting the $7$-th label, we have that $\pi_{7*}h=\sum\alpha_{I,J,K}\pi_{7*}\sigma_{I,J,K}$. Now, $\pi_{7*}\sigma_{I,J,K}$ can be either zero (for example $\pi_{7*}\sigma_{12,34,567}$), or a boundary divisor (for example $\pi_{7*}\sigma_{12,345,67}$). Therefore $\pi_{7*}h$ is an effective sum of boundary divisors on $\overline{M}_{0,6}$. However, $\pi_{7*}h$ can be thought of as the equivalence class of the surface in $\overline{M}_{0,6}$ obtained by embedding the blow up of $\mathbb{P}^2$ at $p_1,,\ldots,p_6$. But then $\pi_{7*}h$ has to be a Keel-Vermeire divisor (see Remark~\ref{kvasembeddedblowups}), implying that $\pi_{7*}h$ cannot be written as an effective sum of boundary $2$-strata. This gives a contradiction.
\end{proof}
\begin{theorem}
\label{firstmaintheorem}
Let $h\in\emph{Eff}_2(\overline{M}_{0,7})$ be the equivalence class of a special hypertree surface on $\overline{M}_{0,7}$. Then $h\notin V_2^{KV}(\overline{M}_{0,7})$.
\end{theorem}
\begin{proof}
Let $p_1,\ldots,p_7\in\mathbb{P}^2$ be the points associated to $h$. Up to relabeling, we can assume that $y=5,6,7$ are such that $p_1,\ldots,\widehat{p}_y,\ldots,p_7$ is a planar realization for an irreducible hypertree on the set $[7]\setminus\{y\}$. Assume by contradiction that we can find nonnegative coefficients $\alpha_{I,J,K}$, $\beta_{ab,m,ij}$ such that
\begin{equation*}
h=\sum\alpha_{I,J,K}\sigma_{I,J,K}+\sum_{\{a,b\}\subset[7]}\sum^{15}\beta_{ab,m,ij}\sigma_{ab,m,ij}^{KV},
\end{equation*}
where $\sum\limits^{15}\beta_{ab,m,ij}\sigma_{ab,m,ij}^{KV}$ runs over the $15$ lifts of the Keel-Vermeire divisors on $D_{ab}$. Fix any coefficient $\beta_{a'b',m',i'j'}$ (so that $a',b',m',i'$ and $j'$ are fixed indices). At least one number among $5,6$ and $7$ is not contained in $\{a',b'\}$. Assume without loss of generality that $7\notin\{a',b'\}$. If we consider the morphism $\pi_{7}\colon\overline{M}_{0,7}\rightarrow\overline{M}_{0,6}$ forgetting the $7$-th label, using Lemma~\ref{risolutore} we obtain that
\begin{equation*}
\pi_{7*}h=\sum\alpha_{I,J,K}\pi_{7*}\sigma_{I,J,K}+\sum_{\{a,b\}\subset[7]}\sum^{15}\beta_{ab,m,ij}\pi_{7*}\sigma_{ab,m,ij}^{KV}=
\end{equation*}
\begin{equation*}
\sum\alpha_{I,J,K}\pi_{7*}\sigma_{I,J,K}+\sum_{b\in[6]}\sum^{15}\beta_{7b,m,ij}\pi_{7*}\sigma_{7b,m,ij}^{KV}+\sum_{\{a,b\}\subset[6]}\sum^{15}\beta_{ab,m,ij}\pi_{7*}\sigma_{ab,m,ij}^{KV}=
\end{equation*}
\begin{equation}
\label{osserva}
\sum\alpha_{I,J,K}\pi_{7*}\sigma_{I,J,K}+\sum_{b\in[6]}\sum^{15}\beta_{7b,m,ij}\delta_{bm,ij}^{KV}+\sum_{\{a,b\}\subset[6]}\sum^{15}\beta_{ab,m,ij}\delta_{ab},
\end{equation}
where $\pi_{7*}h$ is a Keel-Vermeire divisor. The total coefficient of the boundary divisor $\delta_{a'b'}$ in (\ref{osserva}) is equal to a sum $(\ldots+\beta_{a'b',m',i'j'}+\ldots)$, where the terms of the sum are equal to some of the coefficients $\alpha_{I,J,K}$, $\beta_{ab,m,ij}$. The Keel-Vermeire divisors are extremal in $\eff_2(\overline{M}_{0,6})$, therefore the coefficient of $\delta_{a'b'}$ has to be zero. Since the terms in the sum $(\ldots+\beta_{a'b',m',i'j'}+\ldots)=0$ are nonnegative, it follows that $\beta_{a'b',m',i'j'}=0$. But $\beta_{a'b',m',i'j'}$ is arbitrary, so any coefficient $\beta_{ab,m,ij}$ is equal to zero. This implies that $h=\sum\alpha_{I,J,K}\sigma_{I,J,K}\in V_2^{KV}(\overline{M}_{0,7})$, which contradicts Lemma~\ref{tevelevmihadetto}.
\end{proof}
\subsection{Classification of the special hypertree surfaces on $\overline{M}_{0,7}$}
\label{classificationspecialhypertreesurfaces}
Let us find all the possible special hypertree surfaces on $\overline{M}_{0,7}$. We start by fixing a planar realization $p_1,\ldots,p_6\in\mathbb{P}^2$ for the irreducible hypertree given by $\Gamma=\{\{1,4,5\},\{1,3,6\},\{2,3,5\},\{2,4,6\}\}$. We consider permutations of these labels later on. Observe that the points $p_1,\ldots,p_6$ span seven lines: three of them contain exactly two labeled points, and the remaining four contain exactly three labeled points. Let $X$ be the union of these seven lines. If $\textrm{char}(\mathbb{K})\neq2$, this points and lines arrangement is shown below in Figure~\ref{figure3} ($\mathbb{K}$ is our base field).
\begin{center}
\begin{figure}[H]
\centering
\includegraphics[scale=.80]{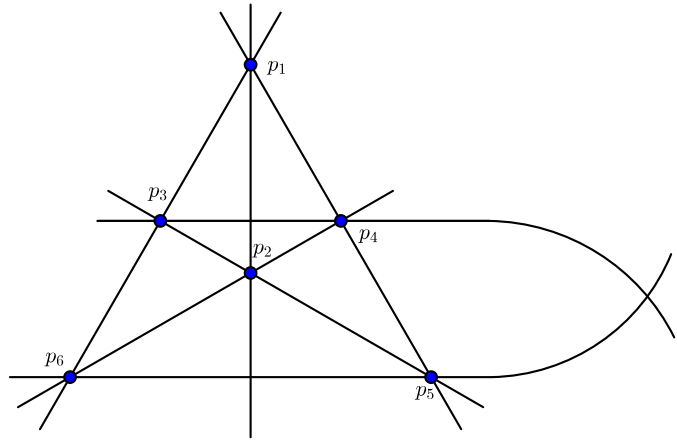}
\caption{Line arrangement spanned by a planar realization for $\Gamma$ if $\textrm{char}(\mathbb{K})\neq2$.}
\label{figure3}
\end{figure}
\end{center}
The characteristic $2$ case is discussed separately at the end of this section. Therefore, for now assume that $\textrm{char}(\mathbb{K})\neq2$.

Let us add a seventh point $p_7$ to the configuration in Figure~\ref{figure3}. Take $p_7\in\mathbb{P}^2\setminus X$. Then we cannot have a special hypertree surface, because if we drop a label $y\in[6]$, the points $p_1,\ldots,\widehat{p}_y,\ldots,p_7$ span at least five lines containing exactly two labeled points. Therefore we must have $p_7\in X$.

Doing similar considerations, one can easily prove that $p_7$ must lie in the intersection of at least two lines in $X$. Since $p_7$ is distinct from $p_1,\ldots,p_6$, we have three possibilities for $p_7$ (the lines in $X$ intersect in $9$ points). All three of these cases give a special hypertree surface, as shown in Figure~\ref{figure4}. The arrows show the three points $p_y$ that can be dropped in order to get an irreducible hypertree on the set $[7]\setminus\{y\}$.
\begin{center}
\begin{figure}[H]
\centering
\includegraphics[scale=.85,valign=t]{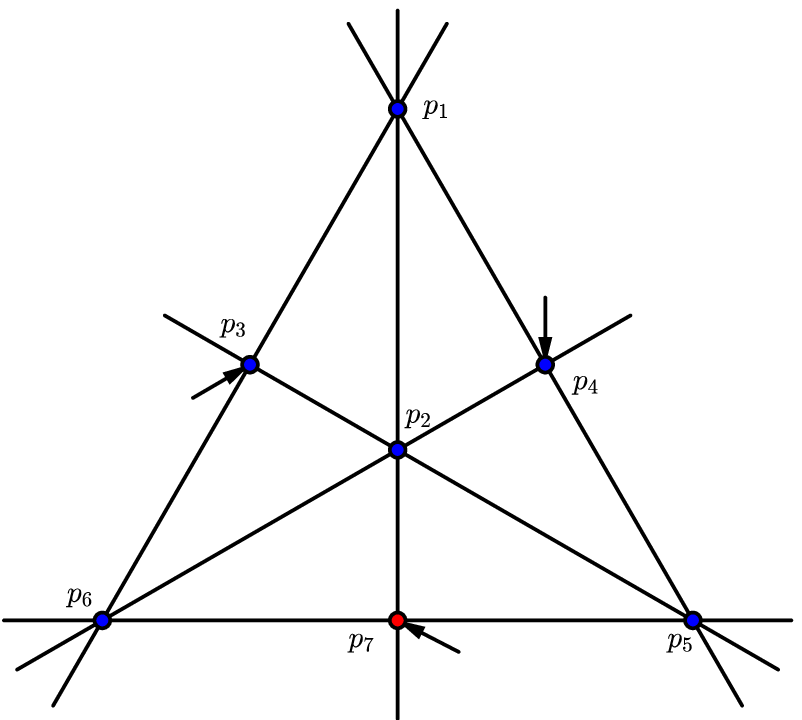}
\includegraphics[scale=.85,valign=t]{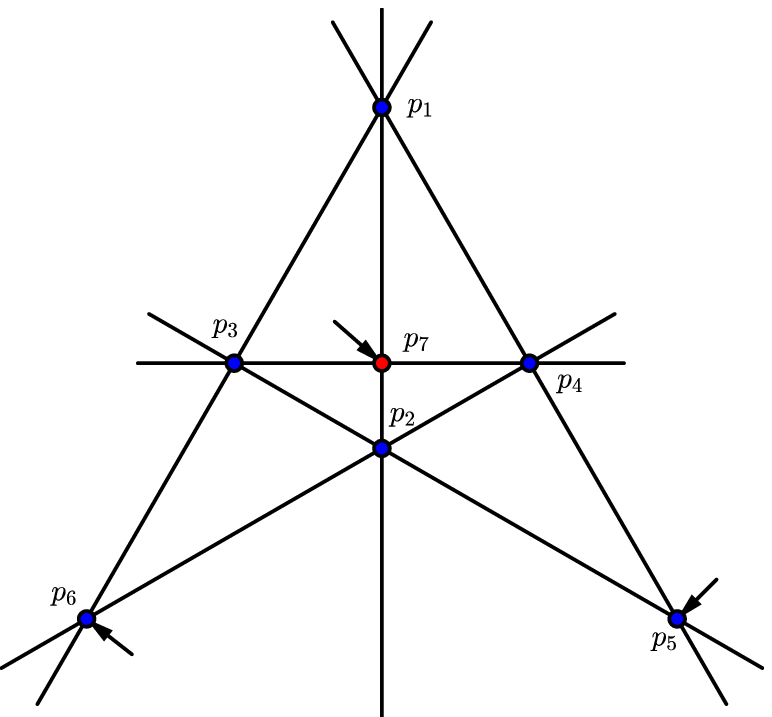}
\includegraphics[scale=.80]{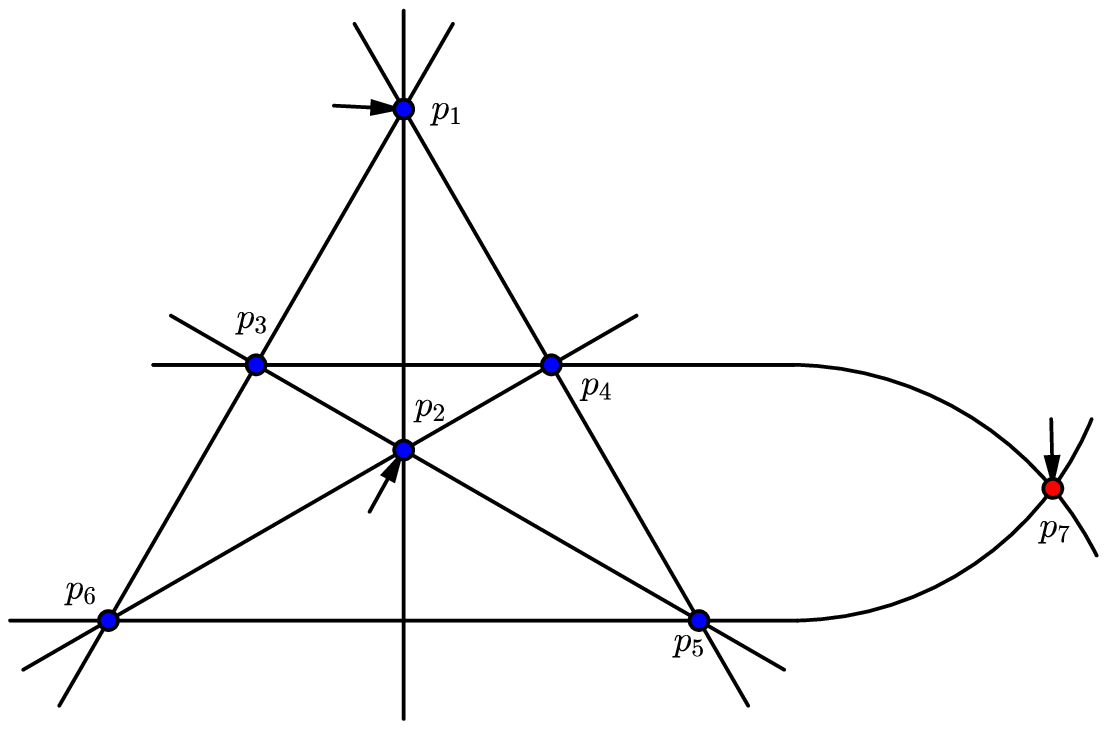}
\caption{Points arrangements in $\mathbb{P}^2$ which give special hypertree surfaces on $\overline{M}_{0,7}$.}
\label{figure4}
\end{figure}
\end{center}
Consider the action $S_7\curvearrowright\eff_2(\overline{M}_{0,7})$ induced by the natural action $S_7\curvearrowright\overline{M}_{0,7}$. When a permutation $\tau\in S_7$ acts on $\sigma\in\eff_2(\overline{M}_{0,7})$, we write $\tau\star\sigma$. Let $h_1$ be the equivalence class of the special hypertree surface obtained by using the top left points configuration in Figure~\ref{figure4}. Similarly, define $h_2$ to be the equivalence class of the special hypertree surface obtained by using the top right configuration, and $h_3$ the one obtained by using the bottom configuration in the same figure.

First, observe that $h_2$ belongs to the orbit of $h_1$ under the $S_7$-action because $h_2=((36)(45))\star h_1$. Also $h_3$ belongs to the orbit of $h_1$, because $h_3=((35)(56)(26)(24)(15)(67))\star h_2$. Therefore, it is enough to consider the $S_7$-action on $h_1$.

Let us find the stabilizer of $h_1$ under the $S_7$-action. It is easy to find the following subgroup of $\textrm{Stab}_{S_7}(h_1)$
\begin{equation*}
G_1:=\{\textrm{id},(34)(65),(47)(16),(37)(15),(156)(473),(165)(374)\},
\end{equation*}
which is isomorphic to the dihedral group $D_3$. It is less obvious to notice this other subgroup of the stabilizer
\begin{equation*}
G_2:=\{\textrm{id},(12)(56),(25)(16),(26)(15)\},
\end{equation*}
which is isomorphic to the Klein group $(\mathbb{Z}/2\mathbb{Z})\times(\mathbb{Z}/2\mathbb{Z})$. To see why $G_1,G_2\subseteq\textrm{Stab}_{S_7}(h_1)$, just take any $\tau\in G_1\cup G_2$ and observe that $\tau\star h_1$ and $h_1$ have the same intersection number with every boundary $2$-stratum on $\overline{M}_{0,7}$.

To show that $\textrm{Stab}_{S_7}(h_1)$ is actually generated by $G_1$ and $G_2$, take any $\tau\in\textrm{Stab}_{S_7}(h_1)$. Thinking of $\tau$ as a bijection $\tau\colon[7]\rightarrow[7]$, then $\tau(\{3,4,7\})=\{3,4,7\}$. This is true because, in order to preserve the intersection numbers with the boundary $2$-strata, we need to send a labeled point that lies on a line containing exactly two labeled points to a labeled point having the same property. In particular, we must have that $\tau(\{1,2,5,6\})=\{1,2,5,6\}$. Therefore $\tau$ acts by permuting the two sets $\{3,4,7\}$ and $\{1,2,5,6\}$ separately. Now there are two cases: $\tau$ fixes $2$ or not. In the first case, the only possibility for $\tau$ is to be an element of $G_1$. If $\tau$ does not fix $2$, then assume that $\tau$ is the identity on $\{3,4,7\}$ (we can assume this up to composing with an element of $G_1$). In this case, one can check that $\tau$ must be an element of $G_2$ in order to preserve the intersection numbers with the boundary $2$-strata on $\overline{M}_{0,7}$. Therefore, we just deduced that $\textrm{Stab}_{S_7}(h_1)=\langle G_1,G_2\rangle$.

An easy count tells us that $\langle G_1,G_2\rangle=24$, and therefore the orbit of $h_1$ has $7!/24=210$ distinct equivalence classes. As one can easily check, these classes generate distinct rays in $\eff_2(\overline{M}_{0,7})$. The next proposition summarizes what we proved so far.
\begin{proposition}
\label{classificationspecialhypertreesurfaceschardifferentfrom2}
In characteristic different from $2$, there are $210$ distinct equivalence classes of special hypertree surfaces on $\overline{M}_{0,7}$. These classes generate $210$ distinct rays of $\emph{Eff}_2(\overline{M}_{0,7})$ which lie outside of the cone $V_2^{KV}(\overline{M}_{0,7})$.
\end{proposition}
\begin{classification in characteristic 2}
The discussion in characteristic $2$ is essentially the same, but with the following exceptions. First of all, in Figure~\ref{figure3}, the seven lines intersect in seven points (one of which is unlabeled), giving the well known Fano configuration. Also in this case, $p_7$ has to be the unlabeled point at the intersection of three lines, and therefore we produced only one special hypertree surface. Now, if we consider the $S_7$-action, it is straightforward to see that the stabilizer of the special hypertree surface we found is $\textrm{PGL}(3,\mathbb{F}_2)$, which has $168$ elements. So, the analogue of Proposition~\ref{classificationspecialhypertreesurfaceschardifferentfrom2} in characteristic $2$ is the following.
\begin{proposition}
In characteristic $2$, there are $30$ distinct equivalence classes of special hypertree surfaces on $\overline{M}_{0,7}$. These classes generate $30$ distinct rays of $\emph{Eff}_2(\overline{M}_{0,7})$ which lie outside of the cone $V_2^{KV}(\overline{M}_{0,7})$.
\end{proposition}
\end{classification in characteristic 2}
\begin{remark}
We do not know yet if the rays generated by the equivalence classes of the special hypertree surfaces are extremal in $\eff_2(\overline{M}_{0,7})$, and certainly a proof or a disproof of the extremality of these rays would be a further step toward the understanding of the cone $\eff_2(\overline{M}_{0,7})$.
\end{remark}
\begin{remark}
We observed that the equivalence classes of the special hypertree surfaces are invariant with respect to a certain subgroup of $S_7$. Given a subgroup $G$ of $S_n$, the idea of considering $G$-invariant sub-loci of $\overline{M}_{0,n}$ intersecting the interior $M_{0,n}$ recently appeared in \cite{moonswinarski}. Also, the same idea was previously used to describe the Keel-Vermeire divisors (see \cite[Section 3]{vermeire}).
\end{remark}
\begin{remark}
Consider the moduli space $\overline{M}_{0,n}^{S_n}$, which is the quotient of $\overline{M}_{0,n}$ by the natural action $S_n\curvearrowright\overline{M}_{0,n}$. As we studied $\eff_2(\overline{M}_{0,7})$, one can also consider $\eff_2(\overline{M}_{0,7}^{S_7})$. For a study of the pseudoeffective cone $\peff_2(\overline{M}_{0,7}^{S_7})$ see \cite[Section 7.3]{fulgerlehmann}.
\end{remark}
Let us define the following subcone of $\eff_2(\overline{M}_{0,7})$.
\begin{definition}
Define $V_2^{KV+CT}(\overline{M}_{0,7})\subseteq\eff_2(\overline{M}_{0,7})$ to be the cone generated by the equivalence classes of the boundary $2$-strata, the lifts of the Keel-Vermeire divisors on $\overline{M}_{0,6}$ and the embedded blow ups of $\mathbb{P}^2$ in $\overline{M}_{0,7}$.
\end{definition}
It follows from what we proved that we have strict inclusions:$$V_2(\overline{M}_{0,7})\subsetneq V_2^{KV}(\overline{M}_{0,7})\subsetneq V_2^{KV+CT}(\overline{M}_{0,7}).$$


\section{Generalization to $\overline{M}_{0,n}$ for any $n>7$ and further questions}
\label{generalization}
We can generalize our constructions for $2$-cycles on $\overline{M}_{0,7}$ to any $\overline{M}_{0,n}$ with $n>7$. First, define $V_2^{KV}(\overline{M}_{0,n})$ inductively to be the subcone of $\eff_2(\overline{M}_{0,n})$ generated by $V_2(\overline{M}_{0,n})$ and by the lifts of the effective $2$-cycles in $V_2^{KV}(\overline{M}_{0,n-1})$. Similarly, we can define $V_2^{KV+CT}(\overline{M}_{0,n})$ inductively to be the subcone of $\eff_2(\overline{M}_{0,n})$ generated by $V_2(\overline{M}_{0,n})$, by the lifts of the effective $2$-cycles in $V_2^{KV+CT}(\overline{M}_{0,n-1})$ and by the embedded blow up of $\mathbb{P}^2$ in $\overline{M}_{0,n}$. Since we already know that $V_2(\overline{M}_{0,7})\subsetneq V_2^{KV}(\overline{M}_{0,7})\subsetneq V_2^{KV+CT}(\overline{M}_{0,7})$, it is not hard to see that we have the following strict inclusions
\begin{equation*}
V_2(\overline{M}_{0,n})\subsetneq V_2^{KV}(\overline{M}_{0,n})\subsetneq V_2^{KV+CT}(\overline{M}_{0,n}).
\end{equation*}

At this point, one can ask the following questions.
\begin{question1}
Is $V_2^{KV+CT}(\overline{M}_{0,7})$ equal to $\eff_{2}(\overline{M}_{0,7})$?
\end{question1}
\begin{question2}
Is it possible to give examples of embedded blow ups of $\mathbb{P}^2$ in $\overline{M}_{0,7}$ that generate extremal rays of $\eff_2(\overline{M}_{0,7})$?
\end{question2}
\begin{question3}
Is $\eff_2(\overline{M}_{0,7})$ equal to $\peff_2(\overline{M}_{0,7})$?
\end{question3}



\

{\small\textsc{Department of Mathematics, University of Georgia, Athens, GA 30602, USA}}

\emph{E-mail address:} \url{luca@math.uga.edu}

\end{document}